\documentclass[11pt, reqno]{amsart}
\usepackage{amssymb}
\usepackage{latexsym}
\usepackage{amsmath}
\usepackage{amsfonts}
\usepackage[normalem]{ulem}
\usepackage[hidelinks]{hyperref}
\hypersetup{
  colorlinks   = true, 
  urlcolor     = blue, 
  linkcolor    = blue, 
  citecolor   = blue 
}
\usepackage{enumitem}
\usepackage{fancyhdr}
\usepackage{cleveref}
\usepackage{mathrsfs}
\usepackage[dvipsnames]{xcolor}
\usepackage[normalem]{ulem}
\usepackage{letterspace}

\usepackage{soul}
\usepackage{tikz}
\usetikzlibrary{arrows.meta}
\usetikzlibrary{arrows}
\usepackage{caption}
\usepackage{bm}
\usepackage{color}
\theoremstyle{plain}
\newtheorem{theorem}{Theorem}[section]

\newtheorem{untheorem}{Theorem}

\newcommand{\R}{\ensuremath{\mathbb{R_+}}}

\newtheorem{lemma}[theorem]{Lemma}
\newtheorem{proposition}[theorem]{Proposition}

\newtheoremstyle{remark}
    {} 
    {} 
    {}          
    {}          
    {\bfseries} 
    {.}         
    {.5em}      
    {}          

\theoremstyle{remark}
\newtheorem{remark}{Remark}[section]

\newcommand{\ROMAN}{\textls*[-100]}

\newtheoremstyle{example}
    {\dimexpr\topsep/2\relax} 
    {\dimexpr\topsep/2\relax} 
    {}          
    {}          
    {\bfseries} 
    {.}         
    {.5em}      
    {}          

\theoremstyle{example}

\newtheoremstyle{definition}
    {\dimexpr\topsep/2\relax} 
    {\dimexpr\topsep/2\relax} 
    {}          
    {}          
    {\bfseries} 
    {.}         
    {.5em}      
    {}          

\theoremstyle{definition}

\newtheorem*{similartheorem*}{Theorem \dualnumber{$'$}}

\numberwithin{equation}{section}

\newcommand{\N}{\mathbb N}
\def\R{\mathbb{R}}
\def\H{\mathbb{H}}
\def\Ha{\mathbb{H}^n}
\def\C{\mathbb{C}}

\newcommand{\Ca}{\mathbb{C}^n}

\newcommand{\inp}[2]{\langle #1, #2\rangle}
\newcommand{\im}{\mathop{\mathrm{Im}}}
\allowdisplaybreaks

\begin{document}
\title[Distances in sets of positive Kor\'anyi upper density in Heisenberg Group]{Distances in sets of positive Kor\'anyi upper density in Heisenberg Group}

\keywords{Heisenberg group, distance sets}
{\let\thefootnote\relax\footnote{\noindent 2010 {\it Mathematics Subject Classification.} Primary 42B20, 26A33; Secondary 43A85}}

\thanks{The second author has been supported by IISER Berhampur and GITAM University, Visakhapatnam.}

\author{K S Senthil Raani and Rajesh K. Singh}

\address[K S Senthil Raani]{IISER Berhampur, Laudigam, District Ganjam, Odisha, India-760003.}
\email{raani@iiserbpr.ac.in}

\address[Rajesh K. Singh]{Department of Mathematics and Statistics, Gitam Institute of Science, GITAM University, Visakhapatnam - 530045, A.P., India}
\email{agsinghraj12@gmail.com, rsingh4@gitam.edu}

\pagestyle{headings}

\begin{abstract}

We prove that any measurable set in the Heisenberg group, $\Ha$, of positive upper density has the property that all sufficiently large real numbers are realised as the Kor\'anyi distance between points in that set. The result can be seen as a Heisenberg group analogue to a corresponding Euclidean large distance set result in the $1986$ paper of Bourgain, \cite{1986Bourgain}.
 
Along the way, to prove our main theorem, we give the ``decay" of the coefficients $R_{k}(\lambda, \sigma)$, appearing in the spectral decomposition of the group Fourier transform, $\hat{\sigma}(\lambda)$ $= \sum_{k=0}^{\infty} R_{k}(\lambda, \sigma) \mathcal{P}_{k}(\lambda)$, of  the surface measure $\sigma$ on the Kor\'anyi sphere in $\Ha$, in a certain ``high frequency" region, that is, when $2(2k+n) |\lambda| \gg 1$; which seems to be new in the literature. We also show  that the positive upper density cannot be qualitatively improved further.
\end{abstract}

\maketitle



\section{Introduction} \label{intro-section}

The structure and richness of distance sets arising from subsets of metric spaces has been a focal point of investigation in geometric measure theory and harmonic analysis. Understanding distance sets has broad implications in analyzing geometric configurations, understanding patterns in large data sets, and studying the structure of fractals. A fundamental question, referred to as Falconer's distance problem, asks: for what $s$ must a set $E \subset \mathbb R^d$ of Hausdorff dimension $s$ ensure that  the set $\Delta_{\mathbb R^d}(E)$, of all Euclidean distances between points in $E$, has positive Lebesgue measure? This problem remains largely open, with known partial results improving incrementally over time. The behavior of analogous distance set problems in non-Euclidean settings remains comparatively underexplored (see for example, \cite{{BF24},{IM05}, {GGPP25}, {ILX22}} and references within for non-Euclidean set up other than Heisenberg group). It is quite interesting to ask for an analogue of Falconer distance set conjecture in the Heisenberg group set-up, given its foundational importance and the rich interplay with the geometry of sets. 

The basic tools to address distance set problems in Euclidean setup are, primarily, that the energy integral (see \cite{M15} for further details) of the sparse sets can be written in terms of the Fourier transform of the measures supported on these sets, and secondarily, the pointwise Fourier decay of the surface measure of the Euclidean sphere. It does not seem to be easy to observe any analogue of the energy integral in terms of the group Fourier transform of generic measures in the Heisenberg group and it is of future interest to the authors. Nevertheless in this article we attempt to state an analogue of the secondary tool, that is to find the right ``decay" of the coefficients in the spectral decomposition of the group Fourier transfrom of the surface measure on the Kor\'anyi sphere in the Heisenberg group (see \Cref{coeff. surface measure}). 

In line with earlier studies focusing on Kakeya-type or Besicovitch-type problems or other geometric measure theoretic problems in the Heisenberg group (see \cite{{BFMT12},{Harris23}, {BDFMT13}, {EJJ20}, {FO23}} and references within), this article aims to initiate the study of distance set problems in $\mathbb H^n$. We believe this will open pathways for further quantitative and structural investigations in the Heisenberg group set-up.

The classical theorem of Steinhaus \cite{Steinhaus} on the Euclidean space, $\R^d$, states that if $E \subset \R^d$ has positive Lebesgue measure, then the set $\Delta_{\mathbb R^d}(E)$ has non-empty interior. Steinhaus' theorem has led to a substantial development in a vast array of subsequent work. In particular, analyzing the distance sets $\Delta_{\mathbb R^d}(E)$  and other configuration sets of a similar nature involving generalized distance functions in the literature (see  \cite{{Piccard-1942}, {M15}, {KL06}} and references within). The generalized version of the Steinhaus' Theorem for any locally compact group $G$ with identity $e$ and a left Haar measure $\nu$ in \cite{Stromberg} states that {\emph{``If $A$ is a $\nu$-measurable subset of $G$ such that $0<\nu(A)<\infty$, then the set $A \cdot A^{-1}=\{yx^{-1}:x,y\in A)\}$ has $e$ in its interior."}}

In the Euclidean case, the quantitative version of Steinhaus' theorem was not evident from the proof in \cite{Steinhaus,Stromberg} and was studied in \cite[Lemma II]{Boardman-1970}. As an easy application, the author in \cite{Boardman-1970} observed that unbounded sets with asymptotic density away from $1/2$ has all  distances. In the Heisenberg group set-up, these results are easily observed. We state these results in \Cref{uniform distances of small sets}.  For completeness, we give the proof in \Cref{uniform distances of small sets proof}, similar to the proof described in \cite[Section 10]{Pramanik-Raani2023}.

The existence of sufficiently large distances in the Euclidean sets of positive upper density was studied independently by Bourgain \cite{1986Bourgain}, Falconer et al \cite{Falconer-Marstrand-1986} and Furstenberg, et al \cite{FKW-1990}. We state their result with a slight modification in their hypothesis as observed in \cite[Section 10.2]{Pramanik-Raani2023}:

\begin{untheorem}[\cite{{1986Bourgain}, {Falconer-Marstrand-1986},{FKW-1990}}] \label{all-large-dist-thm}
Suppose $d \geq 2$, and that $A \subseteq \mathbb R^d$ has positive upper density, that is,
\begin{equation} \label{Bourgain-liminf-condition} 
\limsup_{R \rightarrow \infty} \sup_{x \in \R^d} \frac{\lambda_d(A \cap B(x,R))}{\lambda_d(B(x,R))} > 0,
\end{equation} 
where $B(x,R)$ denotes a ball of radius $R$ centred at $x$ and $\lambda_d$ denotes the $d$-dimensional Lebesgue measure. Then the set $A$ contains all sufficiently large distances; namely, there exists $R_0 = R_0(A) > 0$ with the property that for every $R \geq R_0$, one can find $x, y \in A$ with $|x-y| = R$. Here, $|\cdot|$ denotes the Euclidean norm. 
\end{untheorem}

In this article, our main goal is to find an anologue of \Cref{all-large-dist-thm} in the Heisenberg group set-up. Recall that the Heisenberg group, $\mathbb{H}^n$, is the two step nilpotent Lie group with underlying manifold $\mathbb{C}^n \times \mathbb{R}$ associated with the group law
\begin{equation}
(z,t) \cdot (w,s) := \left(  z + w, t + s + {\textstyle \frac{1}{2} }  \im(  z . \overline{w}) \right),\ \ \text{for all} \ \ (z,t), (w,s) \in \mathbb{H}^n,
\end{equation} 
where $z. \overline{w}= z_{1} \overline{w_{1}}+ \cdots + z_{n} \overline{w_{n}}$, for any $z,w \in \Ca$. We have a family of non-isotropic dilations defined by $\delta_{r}(z,t):=(rz,r^2t)$, for all $(z,t) \in \mathbb{H}^n$, for every $r>0$. The Kor\'anyi norm of $(z,t)$ in $\mathbb H^n$ is given by 
\begin{equation}
    |(z,t)|_K:=\left(|z|^4 + t^2\right)^{\frac14}.  \label{eq: Koranyi}
\end{equation} 
The Kor\'anyi norm is homogeneous of degree 1, that is $|\delta_{r}(z,t)|_K= r \,  |(z,t)|_K$.  The distance between two points $x$ and $y$ in $\mathbb H^n$  will be given by the left invariant metric defined by $d_K(x,y)=| y^{-1} \cdot x|_K$. 
The Haar measure on $\mathbb H^n$ is given by the Lebesgue measure $dz\ dt$ on $\mathbb R^{2n}\times \mathbb R$. The Kor\'anyi ball, $B(a,r)$, of radius $r>0$ centered at $a \in \Ha$ is $ a \cdot B(0,r) = \{x \in \mathbb H^n:|a^{-1} \cdot x|_K<r\} $. One has its measure $|B(a,r)| = C_{Q} \, r^{Q}$, where $Q=(2n + 2)$ is known as the homogeneous dimension of $\H^n$. Let $\Delta(E)=\{d_K(x,y):x,y\in E\}$ denote the distance set of $E\subset \mathbb H^n$. 

We define {\textbf{Kor\'anyi upper density}} of a set $A\subset \mathbb H^n$ as 
\begin{equation} \limsup_{R \rightarrow \infty} \sup_{x \in \mathbb H^n}  \frac{ |A \cap B(x,R)|  }{|B(x,R)|}.\label{defn Koranyi density}
\end{equation}
The main result of this paper is to show that, in the Heisenberg group, the set of positive Kor\'anyi upper density admits large enough distances:
\begin{theorem} \label{all-large-dist-thm Hn}
Suppose $n \geq 1$, and that $A \subseteq \mathbb H^n$ has positive Kor\'anyi upper density. Then the set $A$ contains all sufficiently large distances; namely, there exists $R_0 = R_0(A) > 0$ with the property that for every $R \geq R_0$, one can find $x, y \in A$ with $d_K(x,y)=|x \cdot y^{-1}|_K = R$.
\end{theorem}
Finally, as a form of sharpness of the main \Cref{all-large-dist-thm Hn}, we establish that the Kor\'anyi upper density in the above theorem cannot be replaced by a slower decaying density by constructing an example (see \Cref{Example Rice based}). This example is similar to the Euclidean one constructed in \cite{2020Rice} for the Theorem \ref{all-large-dist-thm}.
\begin{proposition}\label{prop:Example}
Suppose $f:(0,\infty)\rightarrow \mathbb [0,1]$ such that $\lim_{R\rightarrow \infty}f(R)=0$. There exists $A\subset \mathbb H^n$ and a sequence of $R_m\rightarrow \infty$ such that 
    \begin{enumerate}
    \item $\frac{|A\cap B(0, R_m)|}{|B(0, R_m)|}\geq f(R_m),$ for all $m=1,2,3, \dots$
    \item $|x-y|_K\neq R_m,$ for all $x,y\in A$ and  for all $m=1,2,3, \dots$.
    \end{enumerate}
    \end{proposition}
\subsection{On Proof techniques} The study of spherical averaging operators and their associated Fourier transforms has played a central role in Euclidean harmonic analysis and geometric measure theory. A foundational object in this context is the surface measure on the Euclidean sphere, whose Fourier transform exhibits pointwise asymptotic decay that gives a wide range of results, including restriction estimates, distance set problems, and dispersive PDE estimates. The pointwise asymptotics of Fourier transform of the surface measure on the unit sphere in $\mathbb R^n$ are fundamental (see \cite{Ste93}) for understanding the fine-scale structure of convolution operators, oscillatory integrals, and in particular, is the main ingredient of the Fourier analytic proof of Theorem \ref{all-large-dist-thm} by Bourgain in \cite{1986Bourgain}.

A natural analogue of the Euclidean spherical surface measure in the Heisenberg group set up is the surface measure supported on the unit Kor\'anyi Sphere, $S_K(0,1)$, which is defined as 
 \begin{equation}
 S_K(0,1):=\lbrace(z,t) \in \H^n : |(z,t)|_K=1 \rbrace.
 \end{equation}
Let $\sigma$ denote the surface measure on $S_K(0,1)$, normalised to have unit mass, and let $R_{k}(\lambda, \sigma)$ be the coefficients in the spectral decomposition of the group Fourier transform of  $\sigma$ on the Kor\'anyi sphere in $\Ha$: $\hat{\sigma}(\lambda)= \sum_{k=0}^{\infty}  R_{k}(\lambda, \sigma)  \mathcal{P}_{k}(\lambda)$. We refer the reader to \Cref{Preliminaries} for a precise definition and reminder of their basic properties. While previous works have explored various important analytic tools in Heisenberg group, the precise ``decay" of the coefficients $R_k(\lambda,\sigma)$ in the spectral decomposition of the group Fourier transform of the surface measure on the Korányi sphere—that is, the analogue of the Euclidean stationary phase analysis—has not been fully characterized to the authors' best knowledge. Hence it is also of an independent interest to study $R_k(\lambda, \sigma)$. The novelty of the present article lies in providing the right ``decay" of these coefficients $R_{k}(\lambda, \sigma)$ (see Lemma \ref{coeff. surface measure}), thus mirroring the Euclidean setup.

\subsection{Sketch of the proof of the main theorem} Following Bourgain \cite{1986Bourgain}, the proof of our main  \Cref{all-large-dist-thm Hn} is obtained by first establishing the following quantitative analogue in the Heisenberg group.
\begin{theorem}
\label{FKW-quantitative}
Let $0< \epsilon< \frac{1}{2}$, let $E \subset  B(0,1) \subset \H^n$ have measure $|E| \geq \epsilon$, and let $J=J(\epsilon)$ be a sufficiently large natural number depending on $\epsilon$. Suppose that $0 < r_{J} < \cdots < r_{1} \leq 1 $ are a sequence of scales with $r_{j+1} \leq r_{j}/2$ for all $1 \leq j < J$. Then for at least one $1 \leq j \leq J$, one has
\begin{equation}\label{Main inequality}
    \int_{\H^n} \int_{S_{K}(0,1)} 1_{E}(x) 1_{E}( x \cdot \delta_{r_{j}}y ) d \sigma(y) \,  dx \gtrsim \epsilon^{2},
\end{equation}
where, $d \sigma$ is the surface measure on the unit Kor\'anyi sphere $S_{K}(0,1)$, normalised to have unit mass.\footnote{See \Cref{notation-sec} for our conventions on asymptotic notation.} 
\end{theorem}
Intuitively, the reason for this quantitative result, as in the Euclidean case \cite{1986Bourgain}, is based on Fourier analytic techniques. Thus, from Fourier inversion, we write left side of (\ref{Main inequality}) in terms of group Fourier transforms $\widehat{1_{E}}$ of the function $f:=1_{E}$, the indicator function of the set $E$, and $\widehat{\sigma}$ that of the Kor\'anyi surface measure  $\sigma$ as 
\begin{equation}\label{total integral}
    (2\pi)^{-n-1}  \int_{\R^{*}} \sum_{k=0}^{\infty} R_{k}(\lambda r_{j}^{2},\sigma) \sum_{|\alpha|=k} \langle \hat{f}(\lambda)^{*} \hat{f}(\lambda) \Phi_{\alpha}^{\lambda},\Phi_{\alpha}^{\lambda} \rangle \ |\lambda|^n d\lambda,
\end{equation}
which we decompose into ``low", ``middle" and ``high frequencies" determined by the size of $2(2k+n) \, |\lambda| r_{j}^2$. Here, $R_{k}(\lambda, \sigma)$  are the coefficients in the spectral decomposition $$\widehat{\sigma}(\lambda)= \sum_{k=0}^{\infty} R_{k}(\lambda, \sigma) \mathcal{P}_{k}(\lambda),$$ (see Subsection \ref{FT of sigma}) which are given by
\begin{equation*}
    R_{k}(\lambda, \sigma):= c_{n} \frac{\Gamma(k+1)}{\Gamma(k+n)} \int_{- \pi/2}^{\pi/2} L_{k}^{n-1} \left( {\textstyle\frac{1}{2} } |\lambda |    \cos{\theta}  \right) \, e^{-\frac{1}{4}  |\lambda|  \cos{\theta}}  e^{i \frac{1}{4} \lambda \sin{\theta}} (\cos{\theta})^{n-1} d \theta.
\end{equation*}
For the contribution of the low  frequencies to (\ref{total integral}), the factor $R_{k}(\lambda, \sigma)$ being close to $1$ (See \Cref{Surface coeff. low}) so by relating the expression $\sum_{|\alpha|=k} \langle \hat{f}(\lambda)^{*} \hat{f}(\lambda) \Phi_{\alpha}^{\lambda},\Phi_{\alpha}^{\lambda} \rangle$ to the $L^{2}$-norm of the twisted convolution of $f^{\lambda}$ (the Euclidean Fourier transform of $f$ in the last variable, see (\ref{eq:inverseFT})) with the Laguerre functions, we will obtain the lower bound to this quantity that is  $\gtrsim |E|^2 \geq \epsilon^2$ if $2(2k+n)|\lambda|r_{j}^2 \leq \delta$ for some small enough $\delta$, see \Cref{the LFE}.  The contribution of the middle frequencies (that is, when $ \delta \leq 2(2k+n) \, |\lambda| r_{j}^2 \leq 1/ \delta $) to (\ref{total integral}) can be handled by a straightforward adaptation of Bourgain's ``energy pigeonholing" argument used by him in the Euclidean case.

To show that the contribution  of high frequencies  to the integral (\ref{total integral}) is negligible, we need the precise ``decay" of these coefficients $R_{k}(\lambda, \sigma)$, for high frequencies, that is when $ 2(2k+n) |\lambda|$ large. The key novelty of this paper is that we could manage to find the right ``decay" of these  coefficient of $R_{k}(\lambda, \sigma)$ that works for $\H^n$, for all $n \geq 1$. Set $\mu = 2(2k+n)$.
\begin{lemma} \label{coeff. surface measure}
    Let $n \geq 1$. For $\mu |\lambda| \gg 1$, one has
    \begin{equation}
        | R_{k}(\lambda, \sigma) | \lesssim ( |\lambda| \mu )^{-( \frac{n}{2} - \frac{1}{4} )}.
    \end{equation}
\end{lemma}
This ``decay" of these coefficients, $R_{k}(\lambda, \sigma)$,  seems to be new in the literature. Roughly speaking, our approach to achieve it involves first splitting the integral for $R_{k}(\lambda, \sigma)$ into  the integral over the region where the  phase  $\exp{\textstyle i \frac{1}{4} \lambda \sin{ \theta}}$ is stationary plus over the remaining  nonstationary region. To handle the contribution of the stationary region to $R_{k}(\lambda, \sigma)$ one can discard the factor $\exp{\textstyle i \frac{1}{4} \lambda \sin{ \theta}}$ and estimate the rest part of the integral by employing the  delicate asymptotics of the Laguerre polynomials to deal with the further oscillations present due to the factor  $L_{k}^{n-1}( \frac{1}{2} \lambda \cos{ \theta} )$. The contribution of the nonstationary region (with respect to the complex exponential) can be dealt by again using the asymptotics of the Laguerre polynomials and estimating the resulting oscillatory integral using van der Corput lemma. See \Cref{mainnovelproof}, for the detailed discussion.

\subsection{Outline of the paper} We begin with the required preliminaries in \Cref{Preliminaries}. \Cref{mainresultproof} reduces \Cref{all-large-dist-thm Hn} to its quantitative analogue \Cref{FKW-quantitative}. The proof of Theorem \ref{FKW-quantitative} is in \Cref{mainpropproof}. The crucial \Cref{coeff. surface measure} is proven in \Cref{mainnovelproof}. \Cref{appendix} establishes \Cref{prop:Example} along with the quantitative version of Steinhaus's theorem in the Heisenberg group and its consequence.
\subsection{Notation}\label{notation-sec} We use the following asymptotic notation throughout the paper:
\begin{itemize}
    \item  We adopt the notation  $A\lesssim B$ or $A=O(B)$ or $ B \gtrsim A$ to denote the estimate $|A| \leq C B$ for some constant $C$ independent of $A$ and $B$. Here, if $C$ can be chosen small enough, we use the notation $A \ll B$ to mean that $B$ is sufficiently larger than $A$. And, $A \simeq B $ means both $A\lesssim B$ and $ A \gtrsim B$. If we need the constant $C$ to depend on additional parameters (for instance, the dimensional index $n$ of $\Ha$ or some other parameter such $k_{0}$)  we indicate this by subscripts; thus for example, $A \lesssim_{n,k_{0}} B$ or $A =O_{n,k_{0}}( B )$ denotes the bound $|A| \leq C_{n,k_{0}} B$ for some constant $C_{n,k_{0}}$ that depends on $n$ and $k_{0}$.
    \item We use $A \sim B$ as $x \rightarrow \infty$ to denote $A = (1 + E) B$ where $|E| \leq c(x)$ for some function $c(x)$ of a parameter $x$ that goes to zero as {$x \rightarrow \infty$} (holding all other parameters fixed).
    \item $\mu:= 2(2k+n)$, where $k \in \N_{0}$, the set of nonnegative integers, and $n \geq 1$.
\end{itemize}

\section{Preliminaries}\label{Preliminaries}
As our approach is to use Fourier analytic tools (especially in proving the quantitative version of our main theorem) so we gather here all the material relevant for analysis on the Heisenberg group. In particular, we need to recall the facts about Schr\"odinger representations, Weyl transform, Hermite and special Hermite operators and Laguerre functions.
\subsection{Fourier transform on the Heisenberg group}
\label{sec:GFT}
Let us first briefly recall the unitary representation theory of the Heisenberg group, $\Ha$. For our purposes, it will be more convenient to work
with Schr\"odinger representations. For  $ \lambda \in  \R^{*}:=\R \setminus \{0\}$, the \emph{Schr\"odinger representation} $ \pi_\lambda $ acts on the Hilbert space $ L^2( \R^n)$ as follows:
$$ \pi_\lambda(z,t) \varphi(\xi) = e^{i\lambda t} e^{i(x \cdot \xi+ \frac{1}{2}x \cdot y)}\varphi(\xi+y),\,\,\,$$
where $ z = x+iy \in \Ca $ and $ \varphi \in L^2(\R^n)$. One can easily check that these representations are irreducible and unitary. Ignoring the finite dimensional representations of $\Ha$ (as they do not contribute to the Plancherel measure) one has that, up to unitary equivalence, these $\pi_{\lambda}$ accounts for all irreducible unitary representations of $\Ha$ which are nontrivial at the center of $\Ha$, thanks to a theorem of Stone-von Neumann. For a proof we refer to the monograph of Folland \cite{Fo}. In accordance with general theory, these unitary representations $\pi_{\lambda}$ give rise to a group Fourier transform on $L^{1}(\Ha)$. Given $f \in L^{1}(\Ha)$ its group Fourier transform is an operator valued function on $\R^{*}$ obtained by integrating $ f $ against $ \pi_\lambda$:
$$ \hat{f}(\lambda) = \int_{\Ha} f(z,t) \pi_\lambda(z,t)  dz dt .$$  
Clearly, $ \hat{f}(\lambda) $ is a bounded linear operator on $ L^2(\R^n).$ Let $ f^\lambda (z) $ stands for Euclidean inverse Fourier transform of $ f $ in the last variable $t$: 
\begin{equation}
\label{eq:inverseFT}
f^\lambda(z) = \int_{-\infty}^\infty f(z,t) e^{i\lambda t} dt,
\end{equation}
then one can easily see that
\begin{equation}
\label{eq:FTbis}
\widehat{f}(\lambda) = \int_{\C^n} f^\lambda(z) \pi_\lambda(z,0) dz.
\end{equation}
The operator which takes  a function $ g $ on $ \C^n $ into the operator
$$ \int_{\C^n} g(z) \pi_\lambda(z,0) dz $$ is called the \emph{Weyl transform} of $ g $ and is denoted by $ W_\lambda(g)$. Thus, 
\begin{equation}
    \widehat{f}(\lambda) = W_\lambda(f^\lambda).
\end{equation}
From the explicit formula for the kernel of the Weyl tranform and the Plancherel theorem over $\R$ with respect to the last variable, one has that for $g \in L^{1} \cap L^2 (\Ca)$ its Weyl tranform $W_{\lambda}(g)$ is a Hilbert-Schmidt operator satisfying
\begin{equation}\label{Plancherel Weyl}
    \int_{\Ca} |g(z)|^2 dz = (2\pi)^{-n} |\lambda|^n \,  \| W_{\lambda}(g) \|_{\text{HS}}^2.
\end{equation}
This immediately leads to the the following Plancherel formula:
for $ f \in L^1 \cap L^2(\Ha) $ its Fourier transform  is in fact a Hilbert-Schmidt operator and one has 
\begin{equation}\label{Plancherel}
    \int_{\Ha} |f(z,t)|^2 dz dt = (2\pi)^{-n-1} \int_{-\infty}^\infty \|\hat{f}(\lambda)\|_{HS}^2 |\lambda|^n d\lambda .
\end{equation}
Thus,  the Fourier transform extends to be  a unitary operator between $ L^2(\Ha) $ and the Hilbert space of Hilbert-Schmidt operator valued functions  on $ \R $ which are square integrable with respect to the Plancherel measure, $ (2\pi)^{-n-1} |\lambda|^n d\lambda.$

Taking the inverse Fourier transform in the central variable, that is to say \eqref{eq:inverseFT}, is an important tool which is quite often employed in studying problems on $ \H^n$. For instance, it converts the group convolution on $ \H^n $ into the so-called \emph{twisted convolution} on $ \C^n$. Let us recall that the convolution of two absolutely integrable functions $ f,g $ on $\H^n$ is defined in the usual manner as
$$
f*g(x) := \int_{\H^n} f(xy^{-1})g(y) dy, \quad x\in \H^n.
$$
Writing in coordinates $(z,t) \in \Ha$ and taking inverse Fourier transform with respect to  the central variable $t$, it can be shown that
$$
(f*g)^\lambda  (z) = \int_{\C^n} f^\lambda(z-w)g^\lambda(w) e^{ i \frac{\lambda}{2} \im(z . \overline{w})} dw.
$$
Motivated by this, one defines the $\lambda$-twisted convolution of two absolutely integrable functions $f,g$ on $ \Ca$  to be
\begin{equation*}
  f \ast_{\lambda} g   (z) : = \int_{\C^n} f(z-w) g(w) e^{ i \frac{\lambda}{2} \im(z . \overline{w})} dw.
\end{equation*}
From the definitions above, it follows that the identity $ \widehat{f*g}(\lambda) = \widehat{f}(\lambda) \widehat{g}(\lambda)$ yields the relation 
\begin{equation} \label{Weyl on twisted convolution}
     W_\lambda(f^\lambda*_\lambda g^\lambda)
=W_\lambda(f^\lambda)W_\lambda(g^\lambda).
\end{equation}

\subsection{Hermite functions and the Heisenberg group}
\label{sec:Hermite}
Let $\N_{0}$ denotes the set of all nonnegative integers. For each $k\in \N_0$, the Hermite polynomial $H_k(t) $ on $\mathbb R$ is defined by  Rodrigues' formula $H_k(t)=(-1)^k e^{t^2} {d^k\over d t^k} \big(e^{-t^2}\big)$, and  the $L^2$-normalized  Hermite functions
$h_k(t):=(2^k k !  \sqrt{\pi})^{-1/2} H_k(t) e^{-t^2/2}$, $k\in \mathbb N_0$ form an orthonormal basis
of $L^2(\mathbb R)$.   Whereas, in higher dimensions  
 the $n$-dimensional Hermite functions are given by the tensor products of $h_k$: $\Phi_{\alpha}(x):=\prod_{i=1}^n h_{\alpha_i}(x_i), \quad \alpha=(\alpha_1, \cdots, \alpha_n)\ {\in\N_0^n.}$
The set $\{\Phi_{\alpha}\}_{\alpha\in \mathbb N_0^n}$ forms an orthonormal basis of $L^2(\mathbb R^n)$ and the functions $\Phi_{\alpha}$ are eigenfunctions for the Hermite operator with eigenvalue  $2|\alpha|+n$ where $|\alpha|=\sum_{i=1}^n \alpha_i$. 
Thus, for  every $\varphi \in L^2(\mathbb R^n)$  we have  the Hermite expansion
\[ 
\varphi =\sum_{\alpha\in\N_0^n} \inp{\varphi}{\Phi_\alpha}\Phi_\alpha=\sum_{k \in 2\N_0+d} \mathcal{P}_{k} \varphi,
\]
where $\langle \cdot, \cdot \rangle$ is the inner product in $L^2(\R^n)$ and $\mathcal{P}_{k}$ denotes the Hermite spectral projection given by $\mathcal{P}_{k}f : =\sum_{2|\alpha|+n= k}\langle f, \Phi_{\alpha}\rangle\Phi_{\alpha}$, see for instance \cite[Chapter 1.4]{Thangavelu-Heisenberg}.

Since we will be working on  the Heisenberg group $\Ha \equiv \Ca \times \R$ where we employ the trick of taking inverse Fourier transform in the central variable to reduce the matters to $\Ca$, so for us it more suited to consider the scaled Hermite functions on $\R^n$. So, for $\lambda\in \R^*$ and each $ \alpha \in \N_{0}^n $, we introduce the family of scaled Hermite functions
$$
\Phi_\alpha^\lambda(x) := |\lambda|^{\frac{n}{4}}\Phi_\alpha(\sqrt{|\lambda|} \, x), \quad x\in \R^n.
$$
As these $ \Phi_\alpha^\lambda $ also forms an orthonormal  basis of $L^{2}(\R^n)$, so expressing Hilbert Schmidt norm of $\widehat{f}(\lambda)$ in this basis one gets
$ \|\widehat{f}(\lambda)\|_{\operatorname{HS}}^2=\sum_{\alpha\in \N^n}\|\widehat{f}(\lambda) \Phi_\alpha^\lambda \|_{2}^2 $
and hence, by (\ref{Plancherel}), the Plancherel formula can be written as
$$
\int_{\H^n}|f(z,w)|^2\,dz\,dw=\frac{2^{n-1}}{\pi^{n+1}}
\int_{-\infty}^{\infty}\Big(\sum_{\alpha\in \N^n}\|\widehat{f}(\lambda)\Phi_\alpha^\lambda\|_{2}^2\Big)|\lambda|^n\,d\lambda.
$$
Let $\mathcal{P}_{k} (\lambda)$ denotes the orthogonal projection $L^{2}(\R^n)$ onto the finite dimensional subspace spanned by $\Phi_{\alpha}^{\lambda}$,  with $|\alpha|=k$: $ \mathcal{P}_{k}(\lambda)\varphi = \sum_{|\alpha| =k} \langle \varphi,\Phi_\alpha^\lambda \rangle \Phi_\alpha^\lambda$. Then $\Phi_{\alpha}^{\lambda}$ are eigen functions of the scaled Hermite operators  $ H(\lambda) :=  -\Delta+|\lambda|^2 |x|^2 $, indexed by $\lambda \in \R^{*}$, whose spectral decomposition is given by $H(\lambda) = \sum_{k=0}^\infty (2k+n)|\lambda| \mathcal{P}_{k}(\lambda)$.

If a  function $ f $ on $\Ha$ is radial , that is, $f( z,t)=f_{0}(|z|,t)$, for some $f_{0}$ defined on $(0, \infty) \times \R$,  then its Fourier transform $ \widehat{f}(\lambda) $ is a function of the Hermite operator $ H(\lambda)$. To see this one need to bring in the Laguerre functions. Consider the Laguerre functions of type $n-1$ defined on $\Ca$ as
\begin{equation}
\label{eq:Laguerre}
\varphi_{k,\lambda}^{n-1}(z) := L_k^{n-1}\Big( {\textstyle \frac{1}{2} } |\lambda||z|^2\Big)e^{-\frac14 |\lambda||z|^2}.
\end{equation}
Here, $ L_k^{n-1} $ is the $k$-th Laguerre polynomials of type $ (n-1)$. In general, for $\beta > -1$ the Laguerre polynomials of  type $\beta$ are given by the Rodrigues' formula  
\begin{equation*}
   e^{-r} r^{-\beta} L_{k}^{\beta}(r) := \frac{ 1 }{k!}  \frac{d^k}{dr^{k}} ( e^{-r} r^{k+\beta} ),
\end{equation*}
where, $r>0$ and $k \in \N_{0}$.  The collection $\{ \varphi_{k,\lambda}^{n-1}\}_{k=0}^{\infty}$ forms an orthogonal basis for the subspace consisting of  the (usual) radial functions in $ L^2(\C^n).$  

As can be seen from \cite{STU}, for any $L^{2}(\Ca)$-function $g$ one has the following compact form of the  so-called special Hermite expansion $ g(z)=(2\pi)^{-n} |\lambda|^n \sum_{k=0}^\infty  g*_\lambda \varphi_{k,\lambda}^{n-1}(z).$ The connection between the Hermite projections $\mathcal{P}_{k} (\lambda)$ and the Laguerre functions $\varphi_{k, \lambda}^{n-1}$, via the Weyl transform, is given by the following important formula
\begin{equation}
\label{WeylLaguerre}
W_\lambda(\varphi_{k,\lambda}^{n-1}) = (2\pi)^n |\lambda|^{-n} \mathcal{P}_{k} (\lambda).
\end{equation}
For any function $f$ on $\H^n$, we have the expansion $f^{\lambda}(z)=(2\pi)^{-n} |\lambda|^n \sum_{k=0}^\infty  f^{\lambda}*_\lambda \varphi_{k,\lambda}^{n-1}(z)$ 
which, when specialised to radial $f$ on $\Ha$, reduces to simply a Laguerre expansion. 
$$
f^{\lambda}(z)=(2\pi)^{-n} |\lambda|^n \sum_{k=0}^\infty  \mathfrak{c}_k(\lambda, f) \, \varphi_{k,\lambda}^{n-1} (z).
$$
Here, $ \mathfrak{c}_k(\lambda, f) $ are the Laguerre coefficients of the radial function $ f^\lambda $ on $ \C^n$ given by
$$
\mathfrak{c}_k(\lambda, f)=  |\lambda|^{n/2} \,  \frac{\Gamma(k+1) \Gamma(n)}{\Gamma(k+n)}   \int_{\C^n} f^\lambda(z) \, \varphi_{k,\lambda}^{n-1}(z) dz.
$$
Thus, for radial $f$ on $\Ha$, the above Weyl transform argument and the identity \eqref{WeylLaguerre} yields
$$
\widehat{f}(\lambda) = \sum_{k=0}^\infty  \mathfrak{c}_k(\lambda, f) \,  \mathcal{P}_{k} (\lambda).
$$
We record the following useful identity we will need for lower bounding the low frequency part of (\ref{total integral}) : For $f \in L^{2}(\Ha)$ one has
\begin{equation}\label{twisted convol identity}
  \sum_{|\alpha|=k} \langle \widehat{f}(\lambda)^{*} \widehat{f}(\lambda) \Phi_{\alpha}^{\lambda} , \Phi_{\alpha}^{\lambda}  \rangle = (2 \pi)^{-n} |\lambda|^n \,   \| f^{\lambda} \ast_{\lambda} \varphi_{k, \lambda}^{n-1} \|_{2}^2,
\end{equation}
which is a consequence of (\ref{Weyl on twisted convolution}), (\ref{Plancherel Weyl}) and (\ref{WeylLaguerre}).  We refer the reader to \cite{Thangavelu-Heisenberg} or \cite{STU} for a more detailed treatment of these topics.

\subsection{ Fourier transforms of radial measures}\label{FT of sigma}
Let $\nu$ be a finite Borel measure on $\Ha$. The Fourier transform of such $\nu$ are defined similarly: 
\begin{equation*}
    \widehat{\nu}(\lambda) := \int_{\Ha} \pi_{\lambda}(z,t) \, d \nu (z ,t),
\end{equation*}
which is again a bounded linear operator on $L^{2}(\R^n)$. If the measure $\nu$ has more symmetry then one can give good description of its Fourier transform. Consider the natural action of the unitary group, $ U(n) $, on $ \Ha $ given by $ \bm{k} \cdot(z,t) := ( \bm{k} \cdot z,t)$, $ \bm{k} \in U(n) $. This induces an action on functions and measures on the Heisenberg group.  We say that a measure $ \nu $ is radial if it is invariant under the action of $ U(n)$. Similar to the case of radial functions, if a finite Borel measure $\nu$ is also radial then it turns out that its Fourier transform  $\hat{\nu}(\lambda)$ is a function of scaled Hermite operator $ H(\lambda) = -\Delta+\lambda^2 |x|^2$. If  $ H(\lambda) = \sum_{k=0}^\infty (2k+n)|\lambda| \mathcal{P}_{k} (\lambda)$  stands for the spectral decomposition of this operator, then for a radial measure $ \nu $ we have
$$ \widehat{\nu}(\lambda)  = \sum_{k=0}^\infty  R_k(\lambda, \nu) \mathcal{P}_{k} (\lambda).$$ 
Here, $\mathcal{P}_{k} (\lambda)$ are the orthogonal projection of $L^2(\mathbb{R}^n)$ onto the $k^{th}$ eigenspace spanned by scaled Hermite functions $\Phi^{\lambda}_{\alpha}$ for $|\alpha|=k$. The coefficients $ R_k(\lambda,\nu) $ are explicitly given by
$$ R_k(\lambda,\nu)  =  \frac{\Gamma(k+1) \Gamma(n)}{\Gamma(k+n)}       \int_{\Ha}  e^{i\lambda t} \varphi_{k,\lambda}^{n-1}(z) \, d\nu(z,t). $$ 
Recall that there exist a unique Radon measure $\sigma$, the surface measure, on $S_{K}(0,1)$ such that for every integrable function on $\H^n$ we have
\begin{equation*}
    \int_{\H^n} f(z,t) dz \, dt = \int_{0}^{\infty} \int_{S_{K}(0,1)} f(\delta_{r} y) r^{Q-1} d \sigma(y) dr.
\end{equation*}
For $r>0$, one defines the $r$-dilate, $\sigma_{r}$, of $\sigma$ by 
\begin{equation*}
    \langle \sigma_{r} , \phi \rangle := \int_{\Ha} \phi(\delta_{r}(x)) dx, \ \ \ \phi \in C_{c}(\Ha).
\end{equation*}
Then we have $ \widehat{\sigma_r}(\lambda)  = \sum_{k=0}^\infty  R_k(\lambda, \sigma_r) \mathcal{P}_{k} (\lambda)$, where 
\begin{equation}\label{Formula for coeficients of FT of sigma}
   \begin{split}
   R_k(\lambda, \sigma_r) & = \frac{\Gamma\big(\frac{n+1}{2}\big)}{\sqrt{\pi}\Gamma\big(\frac{n}{2}\big)}   \frac{\Gamma(k+1) \Gamma(n)}{\Gamma(k+n)} \int_{-\pi /2}^{\pi/2} \varphi_{k,\lambda}^{n-1}(r \sqrt{\cos \theta}) e^{i \lambda \frac{1}{4}r^2 \sin \theta} \,\,\, (\cos \theta)^{n-1} d \theta \\
     & =   c_{n} \frac{\Gamma(k+1)}{\Gamma(k+n)} \int_{- \pi/2}^{\pi/2} L_{k}^{n-1} \left( {\textstyle\frac{1}{2} } |\lambda |  r^2  \cos{\theta}  \right) \, e^{-\frac{1}{4}  |\lambda| r^2 \cos{\theta}}  e^{i \frac{1}{4} \lambda r^2\sin{\theta}} (\cos{\theta})^{n-1} d \theta.
   \end{split}
\end{equation}
The above integral cannot be evaluated in a closed form. Nevertheless, this representation of $R_{k}(\lambda, \sigma_{r})$ will be used extensively throughout the paper. In some earlier work of Fischer \cite{VF}, this is used to study the spherical means $ f \ast \sigma_r$ and corresponding spherical maximal function.

\subsection{Asymptotic of Laguerre polynomials} To study the behavior of the coefficients $R_k(\lambda, \sigma_r)$ given by (\ref{Formula for coeficients of FT of sigma}) one require the precise asymptotic expression of Laguerre polynomials $L_{k}^{n-1}$. In fact, there is a large body of literature concerning the asymptotic behavior of the Laguerre polynomials. However, for our purpose we will repeatedly use the asymptotic given below in \Cref{Laguerre-Asymp-Our-setup} which is based on  the following relatively simple asymptotic formula for Laguerre polynomial due to Erd\'elyi.
\begin{theorem}[\cite{Askey-Wainger}, p. 697] \label{Laguerre-Asymp}
Let $\alpha \geq 0$.
\begin{enumerate}
    \item [\emph{(\emph{i})}] \hspace*{.7mm} Given $0<b<1$, there exists $k_{0}$ such that if $k \geq k_{0}$ and $0 \leq x \leq b \,  \nu$, then
    \begin{equation} \label{Bessel.regime}
        \begin{split}
            L^{\alpha}_{k}(x) & = \frac{\Gamma(k + \alpha +1)}{ \Gamma(k+1)} 2^{\alpha - 1/2} \nu^{(1 - \alpha)/2} \,  x^{-(\alpha +1)/2} \\
            & \cdot e^{x/2} \left( \frac{  \psi( \frac{x}{\nu} )  }{  \psi^{ \prime} ( \frac{x}{\nu} )  }  \right)^{1/2} \left[ J_{\alpha} \left( \nu \, \psi( \textstyle\frac{x}{\nu} )  \right) +  O \left ( \frac{1}{\nu} \sqrt{\frac{x}{\nu - x}} \,   \widetilde{J_{\alpha}} \left( \nu \, \psi( \textstyle\frac{x}{\nu} )  \right)   \right) \right].
        \end{split}
    \end{equation}
    \item [\emph{(\emph{ii})}] \hspace*{.7mm} Given $a>0$, there exists $k_{0}$ such that if $k \geq k_{0}$ and $x \geq a \,  \nu$, then
    \begin{equation} \label{Bessel.regime}
        \begin{split}
            L^{\alpha}_{k}(x) & = \frac{  (-1)^{k} \pi^{1/2} 2^{5/6} N^{N + 1/6} e^{x/2}  
            }
            { \Gamma(k+1) \left( - \phi^{\prime}(\textstyle\frac{ x}{\nu})  \right)^{1/2} x^{(\alpha +1)/2} e^{N}
            }    \left[ \emph{Ai} \left( - \nu^{2/3} \, \phi( \textstyle\frac{x}{\nu} )  \right) +  O \left( {\frac{1}{x} } \, \widetilde{ \emph{Ai}} \left( - \nu^{2/3} \, \phi( {\textstyle\frac{x}{\nu} } )  \right)   \right) \right],
        \end{split}
    \end{equation}
\end{enumerate}
where, $\nu = 4N = 2 (2k + \alpha + 1)$,
\begin{equation}
    \psi(t) = \frac{1}{2} \left[ (t- t^2)^{1/2} + \sin^{-1}(t^{1/2}) \right], \ \ \ 0 \leq t <1,
\end{equation}
\begin{equation}
    \phi(t)= \left( \frac{3}{4} \right)^{2/3} \, \begin{cases}
                  \left[ \cos^{-1}{t^{1/2}} - (t - t^{2})^{1/2} \right]^{2/3}, & 0<t\leq 1,\\
                 - \left[ (t^2 - t)^{1/2} - \cosh^{-1} {t^{1/2}} \right]^{2/3},  & t>1,
                \end{cases}
\end{equation}
where,
\begin{equation*}
    \widetilde{J_{\alpha}}(u)= \begin{cases}
                 J_{\alpha}(u), & 0 \leq u \ll  1,\\
                 \left( |J_{\alpha}(u)|^2 + |Y_{\alpha}(u)|^2 \right)^{1/2},  & \text{otherwise},
                \end{cases}
\end{equation*}
and
\begin{equation*}
    \widetilde{ \emph{Ai}  }(u)= \begin{cases}
                 \emph{Ai}(u), & u \geq 0,\\
                 \left( |\emph{Ai}(u)|^2 + |\emph{Bi}(u)|^2 \right)^{1/2},  & u<0.
                \end{cases}
\end{equation*}
Here, $J_{\alpha}$ and $Y_{\alpha}$ are Bessel functions of order $\alpha \geq 0$ and $\emph{Ai}$, $\emph{Bi}$ are Airy functions. 
\end{theorem}
In view of the asymptotic of Bessel and Airy functions (see \cite{Leb72}), and the Stirling formula, the  above theorem when specialized to our set up leads to
\begin{lemma}\label{Laguerre-Asymp-Our-setup}
    Let $n \in \mathbb{N}$ and $k \in \mathbb{N}_{0}$. Let $0 < a < b < 1$. Set $\mu := 2 (2k +n)$. Then, there exists $k_{0} \in \mathbb{N}$, large enough, such that for $k \geq k_{0}$, the following hold: 
    \begin{enumerate}
        \item[\emph{(1)}] In the Bessel regime, that is $0 \leq x \leq b \mu$, one has
    \begin{enumerate}
 \item [\emph{(\emph{a})}] \hspace*{.7mm} If $x \,  \leq 1/ \mu$, then
    \begin{equation} \label{less than q by mu}
        \frac{\Gamma(k+1)}{\Gamma(k+n)} \,  L^{n-1}_{k}(x) \, e^{- \frac{1}{2} x} \,  x^{n-1} = C \,  x^{n-1} \left[1 + O \left(  (x \mu )^{1/2} \right) \right].
    \end{equation}  
 \item [\emph{(\emph{b})}] \hspace*{.7mm} If $ 1 / \mu \leq x \, \leq b \mu$, then
    \begin{equation}\label{1 by mu to mu}
       \begin{split}
            \frac{\Gamma(k+1)}{\Gamma(k+n)} \,  L^{n-1}_{k}(x) &  \, e^{- \frac{1}{2} x} \,  x^{n-1} = C \, x^{\frac{n}{2}-1} \mu^{-\frac{n}{2}+1}  \left( \frac{  \psi( \frac{x}{\mu} )  }{  \psi^{ \prime} ( \frac{x}{\mu} )  }  \right)^{1/2} \\
            & \times \left[  \sqrt{  \frac{2}{ \pi \mu \psi({\textstyle\frac{x}{\mu}})}      } \, \cos \left(  \mu \psi(\textstyle\frac{x}{\mu}) - \frac{\pi}{2}(n-1) - \frac{\pi}{4}      \right) + O \left(  (x \mu )^{-3/4} \right) \right]\\
           & \hspace*{5em}= x^{n-1} O \left(  (x \mu )^{-\frac{2n-1}{4}} \right).
       \end{split}
    \end{equation}      
    \end{enumerate}
     \item[\emph{(2)}] In the Airy regime, that is $a \mu \leq x \leq \frac{1}{a} \mu$,  one has
     \begin{enumerate}
 \item [\emph{(\emph{a})}] \hspace*{.7mm} If $ a \mu \leq x \leq  \mu - c \, \mu^{1/3}$, $0 < c \ll 1$, then    
 \begin{equation} \label{Bessel-Airy interface}
        \begin{split}
            \frac{\Gamma(k+1)}{\Gamma(k+n)} \,   & L^{n-1}_{k}(x) \, e^{- \frac{1}{2} x} \,  x^{n-1} \\
            & = (-1)^k \mu^{- \frac{4}{3} } \mu^{\frac{1}{12}} (\mu - x)^{-\frac{1}{4}}   \left[  \cos \left( {  \textstyle\frac{2}{3} |\xi(x)|^{3/2}  - \frac{\pi}{4} }     \right) + O \left( { \textstyle  \mu^{  \frac{1}{6}} \left( 1 - \frac{x}{\mu}  \right)^{- \frac{1}{2}} } \right)  \right],
        \end{split}
    \end{equation} 
    Here, $\xi(x) :=  - \mu^{2/3} \phi({\textstyle \frac{x}{\mu}})$. Moreover,    $ \xi(x) \sim - \mu^{2/3} \left( 1 - {\textstyle \frac{x}{\mu}}  \right)$, if $a \mu \leq x \leq \mu$ and $ \xi(x) \sim  \mu^{2/3} \left(  {\textstyle \frac{x}{\mu}}  -1 \right)$, if $x > \mu$.
 \item [\emph{(\emph{b})}] \hspace*{.7mm} If $ \mu - c \, \mu^{1/3} \leq x \leq \mu + c \, \mu^{1/3}$, $0 < c \ll 1$, then
    \begin{equation} \label{around mu}
        \begin{split}
            \frac{\Gamma(k+1)}{\Gamma(k+n)} \,   & L^{n-1}_{k}(x) \, e^{- \frac{1}{2} x} \,  x^{n-1} \\
            & = \, (-1)^{k} \mu^{{-\textstyle\frac{n}{2}} - \frac{7}{12} }  x^{ {\textstyle\frac{n}{2} - \frac{3}{4}} } \left[ \emph{Ai}( \xi(x))   +  O \left( {\textstyle \frac{1}{x} } \, \emph{Ai}( \xi(x)) \right)  \right]\\
            &= O \big(  \mu^{{-4/3 }  } \big).
        \end{split}
    \end{equation}  
    
 \item [\emph{(\emph{c})}] \hspace*{.7mm} If $ \mu + c \, \mu^{1/3} \leq x \, \leq {\textstyle \frac{3}{2}} \mu$, then
    \begin{equation}\label{Airy merging with exponential}
        \begin{split}
            \frac{\Gamma(k+1)}{\Gamma(k+n)} \,   & L^{n-1}_{k}(x) \, e^{- \frac{1}{2} x} \,  x^{n-1} \\
            & = \, (-1)^{k} \mu^{{-\textstyle\frac{4}{3}} } \,   \mu^{ {\textstyle\frac{1}{12} } }  (x- \mu)^{ - {\textstyle\frac{1}{4}  } } \, e^{ - {\frac{2}{3}  } \mu^{- \frac{1}{2}} (x- \mu)^{3/2}    } \left[   1 + O \left(   \mu^{-\frac{1}{4}} (x - \mu )^{- \frac{3}{4}}  \right)    \right].
        \end{split}
    \end{equation}  
    \end{enumerate}
    \item[\emph{(3)}] In the exponential regime, that is $x \geq \frac{3}{2} \mu$, one has
     \begin{equation}\label{expo regime}
        \begin{split}
            \frac{\Gamma(k+1)}{\Gamma(k+n)} \,   & L^{n-1}_{k}(x) \, e^{- \frac{1}{2} x} \,  x^{n-1} = O \left( \mu^{ -\frac{n}{2}-\frac{1}{2}} x^{\frac{n}{2}-1} e^{- \frac{2}{3} x}  \right).
        \end{split}
    \end{equation} 
\end{enumerate}
Here, the implied constants in the asymptotic notation may depend on $n, k_{0}$. 
\end{lemma}

\section{Proof of Theorem \ref{all-large-dist-thm Hn}: Reduction to quantitative formulation}\label{mainresultproof}
In this section, we deduce \Cref{all-large-dist-thm Hn} from \Cref{FKW-quantitative}. This deduction  is straightforward extension of the corresponding Euclidean argument by Bourgain. Assume for the sake of contradiction that \Cref{all-large-dist-thm Hn} failed, then one can find a set $A \subseteq \mathbb{H}^n$ of positive upper density $\delta>0$, and a sequence of scales $l_{1}< l_{2}< \cdots$ going to  infinity such that for each $l_{j}$ there are no $x,y \in A$ with $|x \cdot y^{-1}|_{K}= l_{j}$. Sparsifying this sequence of scales so that we have $l_{j+1} \geq 2l_{j}$ for all $j$. Let $\epsilon>0$ be sufficiently small (depending only on $\delta$) and let $J$ be as in \Cref{FKW-quantitative}. As $A$ has density $\delta$, one can find $x_0\in \mathbb H^n$ and  a radius $R> l_{J}$ so that 
\begin{equation*}
    \frac{|A \cap B(x_0,R)|}{|B(x_0,R)|}\geq \frac{1}{2} \delta,
\end{equation*}
or equivalently
\begin{equation*}
    |A \cap B(x_0,R)| \gtrsim \delta \, R^{Q}, \ \ \ Q=2n+2.
\end{equation*}
If we consider the rescaling $E=\lbrace x \in B(0,1) : \delta_{R}x\cdot x_0 \in A  \rbrace = B(0,1) \cap \delta_{1/R}\left( x_0^{-1} \cdot A \right)$, and define $0< r_{J}<\cdots <r_{1} \leq 1$ by
\begin{equation*}
    r_{j}:= l_{J+1-j} \Big{/} R,
\end{equation*}
for $j=1, \cdots, J$, then $|E| \gtrsim \delta  \ (\gg \epsilon)$ and $r_{j+1}= \frac{1}{R} l_{J-j} \leq \frac{1}{R} \frac{1}{2} l_{J-j+1}= \frac{1}{2} r_{j}$, for all $1 \leq j \leq J$; and for any $1 \leq j \leq J$ there are no points $x, y \in E$ so that $d_K(x,y)=r_{j}$.  In particular, the left hand side of (\ref{Main inequality}) vanishes for all $1 \leq j \leq J$,  contradicting \Cref{FKW-quantitative}.

\section{Proof of Theorem \ref{FKW-quantitative}}\label{mainpropproof}
Fix $0 < \epsilon < 1/2$. Let us also fix $E$ to be a Lebesgue measurable set inside $B(0,1) \subset \H^n$ such that $|E| \geq \epsilon$. Set $f=1_{E}$, the indicator function of the set $E$. A short Fourier-analytic calculation reveals that the integral in (\ref{Main inequality}), at the radius $r$, may be written as
\begin{equation*}
    \begin{split}
        \int_{\H^n} & \int_{S_{K}(0,1)} 1_{E}(x) 1_{E}( x \cdot \delta_{r}y ) \, d \sigma(y) \,  dx 
     = (2\pi)^{-n-1}  \int_{\R^{*}} \sum_{k=0}^{\infty} R_{k}(\lambda r^{2},\sigma) \sum_{|\alpha|=k} \langle \hat{f}(\lambda)^{*} \hat{f}(\lambda) \Phi_{\alpha}^{\lambda},\Phi_{\alpha}^{\lambda} \rangle \,  |\lambda|^n d\lambda,
    \end{split}
\end{equation*}
where $R_{k}(\lambda, \sigma)$ are the coefficients of the group Fourier transform of the surface measure on the Kor\'anyi sphere : $\hat{\sigma}(\lambda)= \sum_{k=0}^{\infty} R_{k}(\lambda, \sigma) P_{k}(\lambda)$, which from (\ref{Formula for coeficients of FT of sigma}) are given by
\begin{equation*}
    R_{k}(\lambda, \sigma)=   c_{n} \frac{\Gamma(k+1)}{\Gamma(k+n)} \int_{- \pi/2}^{\pi/2} L_{k}^{n-1} \left( {\textstyle\frac{1}{2} } |\lambda |    \cos{\theta}  \right) \, e^{-\frac{1}{4}  |\lambda|  \cos{\theta}}  e^{i \frac{1}{4} \lambda \sin{\theta}} (\cos{\theta})^{n-1} d \theta.
\end{equation*}
Let $J=J(\epsilon)$ be a large positive integer depending only on $\epsilon$. Given any sequence of sufficiently small scales $0 < r_{J} < \cdots < r_{1} \leq 1 $ depending on all previous ones and $\epsilon$ as $r_{j+1} \leq r_{j}/2$ for all $1 \leq j < J$, we need to ensure that  for at least one $1 \leq j \leq J$, one has
\begin{equation}\label{euivalent main ineq}
   \int_{\R^{*}} \sum_{k=0}^{\infty} R_{k}(\lambda r_{j}^{2},\sigma) \sum_{|\alpha|=k} \langle \hat{f}(\lambda)^{*} \hat{f}(\lambda) \Phi_{\alpha}^{\lambda},\Phi_{\alpha}^{\lambda} \rangle \ |\lambda|^n d \lambda \gtrsim \epsilon^{2}.
\end{equation}
As in the 1986 paper of Bourgain, \cite{1986Bourgain}, we will follow the energy pigeonholing argument. Thus, for $ r_{j}^2 < \delta \ll 1$ to be specified later,  we can split the the integral in (\ref{euivalent main ineq}) into the contribution of the  ``low frequencies"
\begin{equation}\label{Low}
    \int_{\R^{*}} \sum_{k \, : \, 2(2k+n) |\lambda | r_{j}^2 \leq \delta } R_{k}(\lambda r_{j}^{2},\sigma) \sum_{|\alpha|=k} \langle \hat{f}(\lambda)^{*} \hat{f}(\lambda) \Phi_{\alpha}^{\lambda},\Phi_{\alpha}^{\lambda} \rangle \ |\lambda|^n d\lambda,
\end{equation}
the ``medium frequencies"
\begin{equation}\label{Medium}
    \int_{\R^{*}} \sum_{k \, : \, \delta / r_{j}^{2}  \leq 2(2k+n) |\lambda |  \leq 1/ (\delta r_{j}^2) } R_{k}(\lambda r_{j}^{2},\sigma) \sum_{|\alpha|=k} \langle \hat{f}(\lambda)^{*} \hat{f}(\lambda) \Phi_{\alpha}^{\lambda},\Phi_{\alpha}^{\lambda} \rangle \ |\lambda|^n d\lambda,
\end{equation}
and the ``high frequencies"
\begin{equation}\label{High}
    \int_{\R^{*}} \sum_{k \, : \,  2(2k+n) |\lambda |  >  1/ (\delta r_{j}^2) } R_{k}(\lambda r_{j}^{2},\sigma) \sum_{|\alpha|=k} \langle \hat{f}(\lambda)^{*} \hat{f}(\lambda) \Phi_{\alpha}^{\lambda},\Phi_{\alpha}^{\lambda} \rangle \ |\lambda|^n d\lambda.
\end{equation}

\subsection{Low frequency part } \label{Low part}
In this subsection we tackle the low frequency part of the expression in the left hand side of (\ref{euivalent main ineq}), namely, the region where $\mu |\lambda|   r_{j}^2 \leq \delta  \ll 1$, where $\mu = 2(2k+n)$. We begin with observing the behavior of the coefficients of the surface measure in this low frequency case.
\begin{lemma}\label{Surface coeff. low}
    Let $\delta \ll 1$. If $\mu |\lambda| \leq \delta \ll 1$ then one has 
    \begin{equation}\label{Low frequency surface coeff}
        R_{k}(\lambda, \sigma) =  C_{n} + O \big( (\mu |\lambda|)^{1/2} \big),
    \end{equation}
    for some $C_{n}>0$, where the implied constant in the asymptotic notation  depends on $n$.
\end{lemma}
\begin{proof}
By symmetry, we can assume $\lambda>0$ and re-write $R_{k}(\lambda, \sigma)$ as 
\begin{equation*}\label{0 Coeff. Rk as Laguerre integral}
   \begin{split}
        R_{k}(\lambda, \sigma) 
        & = c_{n} \frac{\Gamma(k+1)}{\Gamma(k+n)} \int_{- \pi/2}^{\pi/2} L_{k}^{n-1} \left( {\textstyle\frac{1}{2} } \lambda  \cos{\theta}  \right) \, e^{-\frac{1}{4} \lambda \cos{\theta}}  \cos{ \big( {\textstyle \frac{1}{4} } \lambda \sin{\theta} \big)}     (\cos{\theta})^{n-1} d \theta.
   \end{split}
\end{equation*}
Since here $0 < \lambda \ll 1$ so it suffices to show
\begin{equation}\label{0 Coeff. Rk as Laguerre integral}
   \begin{split}
       \frac{\Gamma(k+1)}{\Gamma(k+n)} \int_{- \pi/2}^{\pi/2} L_{k}^{n-1} \left( {\textstyle\frac{1}{2} } \lambda  \cos{\theta}  \right) \, e^{-\frac{1}{4} \lambda \cos{\theta}}      (\cos{\theta})^{n-1} d \theta =  C_{n} + O \big( (\mu |\lambda|)^{1/2} \big),
   \end{split}
\end{equation}
where the implied constants in the asymptotic notation are allowed to depend on $n$.

\noindent Fix $k_{0}$ to be the large natural number as in the \Cref{Laguerre-Asymp-Our-setup}. Making a change of variable  $x = \frac{1}{2} \lambda \cos{\theta}$, the left side of (\ref{0 Coeff. Rk as Laguerre integral}) becomes 
\begin{equation}\label{0 x var. lambda less than 1 Coeff. Rk}
   \begin{split}
        \lambda^{-n +1} \frac{\Gamma(k+1)}{\Gamma(k+n)}  \int_{0}^{ \lambda /2}  L_{k}^{n-1}(x) \, e^{-\frac{1}{2} x }  x^{n-1} \frac{dx}{\sqrt{(\lambda/2)^2 -  x^2}},
   \end{split}
\end{equation}
which, being in the Bessel regime $0 \leq  x \leq \lambda \ll 1 / \mu$, using the asymptotics (\ref{less than q by mu}) for  $k \geq k_{0}$,  can be written as
\begin{equation*}\label{-1 x var. lambda less than 1 Coeff. Rk}
   \begin{split}
     C_{n} \,   \lambda^{-n +1}   \int_{0}^{ \lambda /2}  x^{n-1} \left[1 + O_{k_{0}} \left(  (x \mu )^{1/2} \right) \right] \frac{dx}{\sqrt{(\lambda/2)^2 -  x^2}}.
   \end{split}
\end{equation*}
By rescaling, the last expression equals a positive constant (depending only on $n$) times the expression
\begin{equation*}
    \int_{0}^{1} \frac{x^{n-1}}{\sqrt{1 - x^2}} dx + O_{k_{0}}( (\lambda \mu)^{1/2}),
\end{equation*}
which {leads the desired estimate (\ref{0 Coeff. Rk as Laguerre integral}) but for $k \geq k_{0}$.}  

The case of bounded $k$'s easily follows from the behavior (near $0$) of the explicit expression of Laguerre polynomials. Indeed, if $k \leq k_{0}$, then we are essentially in the region $0 < \lambda \ll_{k_{0}} \delta$. Thus, for  Laguerre polynomials, $L_{k}^{n-1}$, with their degrees bounded by $k_{0}$, one has $L_{k}^{n-1}(x) = \frac{\Gamma(k+n)}{\Gamma(k+1) \Gamma(n)} + O_{k_{0}}(x)$ valid for small $x$ uniformly for all $0 \leq k \leq k_{0}$. Substituting this into
(\ref{0 x var. lambda less than 1 Coeff. Rk}) one has that for $\lambda$ sufficiently small (\ref{0 x var. lambda less than 1 Coeff. Rk}) becomes a positive constant (depending on $n, k_{0}$) times the expression
\begin{equation}\label{0 x var. lambda less than 1 Coeff. Rk 1}
   \begin{split}
        \lambda^{-n +1}  \int_{0}^{ \lambda /2}  [1 + O_{k_{0}}(x)] \, e^{-\frac{1}{2} x }  x^{n-1} \frac{dx}{\sqrt{(\lambda/2)^2 -  x^2}},
   \end{split}
\end{equation}
which yields (\ref{0 Coeff. Rk as Laguerre integral}) for all $k \leq k_{0}$. 

Altogether, these gives the claimed estimate (\ref{0 Coeff. Rk as Laguerre integral}) and hence (\ref{Low frequency surface coeff}), and the lemma is proved.
\end{proof}

We now return to the low frequency estimate
\begin{lemma}\label{the LFE} For $1 \gg \delta \geq r^{2}$, there exists an absolute constant $C_{n}$ such that
\begin{equation}\label{low frequency contri}
    \int_{\R^{*}} \sum_{k \, : \, 2(2k +n) |\lambda | r^2 \leq \delta } R_{k}(\lambda r^{2},\sigma) \sum_{|\alpha|=k} \langle \hat{f}(\lambda)^{*} \hat{f}(\lambda) \Phi_{\alpha}^{\lambda},\Phi_{\alpha}^{\lambda} \rangle \ |\lambda|^n d\lambda \geq C_{n} |E|^2.
\end{equation}
\end{lemma}
\begin{proof}
Let us abbreviate  $\mu = 2(2k +n)$. We recall that $f= 1_{E}$ as well as the identity 
\begin{equation}\label{twisted convol identity1}
  \sum_{|\alpha|=k} \langle \widehat{f}(\lambda)^{*} \widehat{f}(\lambda) \Phi_{\alpha}^{\lambda} , \Phi_{\alpha}^{\lambda}  \rangle = (2 \pi)^{-n} |\lambda|^n \,   \| f^{\lambda} \ast_{\lambda} \varphi_{k, \lambda}^{n-1} \|_{2}^2,
\end{equation}
which is just (\ref{twisted convol identity}).

Note that,  in view of \Cref{Surface coeff. low} (applied at $\lambda r^2$ instead of $\lambda$) and the identity \eqref{twisted convol identity1}, there exists a positive constant $C_{n}$ so that
\begin{eqnarray}
\nonumber && \left|\int_{\R^{*}} \sum_{k \, : \, 2(2k +n) |\lambda | r^2 \leq \delta} \left(R_{k}(\lambda r^{2},\sigma) -C_n\right)\sum_{|\alpha|=k} \langle \hat{f}(\lambda)^{*} \hat{f}(\lambda) \Phi_{\alpha}^{\lambda},\Phi_{\alpha}^{\lambda} \rangle \ |\lambda|^n d\lambda\right|\\
&&\leq C_{n} \sqrt{\delta} \int_{\R^{*}} \sum_{k \, : \, 2(2k +n) |\lambda | r^2 \leq \delta}\| f^{\lambda} \ast_{\lambda} \varphi_{k, \lambda}^{n-1} \|_{2}^2 \ |\lambda|^{2n} d\lambda \label{eq:CompuLowThorough}
\end{eqnarray}
Hence, with an appropriate choice of $\delta<\frac14$, we see from  \eqref{twisted convol identity1} and \eqref{eq:CompuLowThorough} that the left-hand  side of \eqref{low frequency contri} is  bounded below by
    \begin{equation}\label{low frequency contri 1}
    \frac12C_n\int_{\R^{*}} \sum_{k \, : \, \mu |\lambda | r^2 \leq \delta } \,   \| f^{\lambda} \ast_{\lambda} \varphi_{k, \lambda}^{n-1} \|_{2}^2 \ |\lambda|^{2n} d\lambda.
\end{equation}
Let $\lambda>0$. We now concentrate on the factor $\| f^{\lambda} \ast_{\lambda} \varphi_{k, \lambda}^{n-1} \|_{2}^2$ in the above expression. For this we shall prove that if  $\mu \lambda \ll 1$ for $k \geq k_{0}$, with $k_{0}$ being a large natural number as in \Cref{Laguerre-Asymp-Our-setup}, and $| z| \leq \rho_{k}$ (for some $\rho_{k} \ll 1$ to be optimized later) we have
\begin{equation}\label{twisted term near 0}
    f^{\lambda} \ast_{\lambda} \varphi_{k, \lambda}^{n-1} \, (z) = C_{n}^{\prime} \, \mu^{n-1} |E| +   O (  \mu^{n-1} (\lambda \mu)^{1/2} |E|).
\end{equation}
Note that for $\delta \geq r^2$ the approximation \eqref{twisted term near 0} gives the lower bound for (\ref{low frequency contri 1}). Indeed, for appropriate positive constants $a_n,b_n$,  the expression \eqref{low frequency contri 1} is

\begin{equation*}
    \begin{split}
      & \geq \frac{1}{2} C_{n} \int_{\R^{*}} \sum_{ k \geq k_{0} \, : \,  \mu \lambda \ll 1  } \,   \int_{z \in \Ca : |z | \leq \rho_{k}} |f^{\lambda} \ast_{\lambda} \varphi_{k, \lambda}^{n-1} \, (z)|^2 dz \, \lambda^{2n} d\lambda   \\
       & \geq \frac18 C_{n}{C_n^{\prime}}^2 |E|^2 \,   \int_{0 <\lambda \,  \leq  b_n} \sum_{k_{0} \leq  k \leq a_{n} \lambda^{-1}   } \,   \rho_{k}^{2n} \mu^{2n-2} \ \lambda^{2n} d\lambda.
    \end{split}
\end{equation*}
Now $\rho_{k}$ can be chosen to be $\mu^{-M}$, for some $M>0$, and the last expression then becomes
\begin{equation*}
    \begin{split}       
       &  \frac18 C_{n}{C_n^{\prime}}^2 |E|^2 \,   \int_{0 <\lambda \,  \leq  b_n} \sum_{k_{0} \leq  k \leq a_{n} \lambda^{-1}   } \,  \mu^{-2n M + 2n-2} \ \lambda^{2n} d\lambda,
\end{split}
\end{equation*}
which is $\simeq_{n, k_{0}} |E|^2$. This gives the asserted lower bound for \eqref{low frequency contri 1}, and hence \eqref{low frequency contri} is proved. 

It remains to establish \eqref{twisted term near 0}. Here we use asymptotic of Laguerre polynomials in the Bessel regime and write  
    \begin{equation*}\label{3 twisted term near 0}
    \begin{split}
        f^{\lambda} \ast_{\lambda} \varphi_{k, \lambda}^{n-1} \, (z) &  = \int_{\mathbb H^n} 1_{E}(z-w,-t) \, L_k^{n-1}\big( {\textstyle \frac{1}{2}} \lambda|w|^2 \big) \, e^{- \frac{1}{4}\lambda |w|^2}   e^{i \lambda \left(-t+\frac12 \Im(z\cdot \overline{w}) \right)}dwdt \\
        & = \int_{(z,0) - E \,  \subseteq \mathbb H^n}  L_k^{n-1}\big( {\textstyle \frac{1}{2}} \lambda|w|^2 \big) \, e^{- \frac{1}{4}\lambda |w|^2}   e^{i \lambda \left(-t+\frac12 \Im(z\cdot \overline{w}) \right)}dwdt
    \end{split}
\end{equation*}
where,  we have used that  $1_{E}(z-w, -t)=1$ if and only if $(w,t) \in (z,0) - E$. Since $E \subseteq B(0,1) \subset \Ha$, the Kor\'anyi norm $|(w-z,t)|_{K} <1$, whence  $\frac{\lambda}{4} |w |^2 \leq \lambda$ and $| \lambda \left(-t+\frac12 \Im(z\cdot \overline{w}) \right) | \leq 2\lambda$. Since $\lambda \ll  1$, so $ \left| \exp{ i \lambda \left(-t+\frac12 \Im(z\cdot \overline{w}) \right) } - 1 \right| \leq 4 \lambda$ which together with the estimate $\frac{\Gamma(k+1)}{\Gamma(n)\Gamma(k+n)} |L_{k}^{n-1}(x) e^{-x/2}| \leq 1$,  and the Stirling approximation, gives that the last integral equals
\begin{equation}\label{4 twisted term near 0}
    \begin{split}
         \int_{(z,0) - E \,  \subseteq \mathbb H^n}  L_k^{n-1}\big( {\textstyle \frac{1}{2}} \lambda|w|^2 \big) \, e^{- \frac{1}{4}\lambda |w|^2}  dwdt \ + \ O \left( \lambda \mu^{n-1} |E| \right).
    \end{split}
\end{equation}
Since, ${\textstyle\frac{1}{2}} \lambda |w|^2 \leq 2 \lambda \ll 1/ \mu$  in the integral in (\ref{4 twisted term near 0}), using Laguerre polynomial asymptotic (\ref{less than q by mu}) in conjugation with the Stirling approximation the last expression (\ref{4 twisted term near 0})  becomes
\begin{equation*}\label{5 twisted term near 0}
    \begin{split}
        & C  \int_{(z,0) - E \,  \subseteq \mathbb H^n} \frac{\Gamma(k+n)}{\Gamma(k+1)} \left[ 1 + O  \left( \left( {\textstyle\frac{1}{2}} \lambda |w|^2 \, \mu \right)^{1/2} \right) \right]  dwdt \ + \ O \left( \lambda \mu^{n-1} |E| \right)\\
        & \simeq \mu^{n-1} |E| +  O (  \mu^{n-1} (\lambda \mu)^{1/2} |E|),
    \end{split}
\end{equation*}
which is what we were aiming for in (\ref{twisted term near 0}). Thus, we have completed the proof of \Cref{the LFE}.
\end{proof}

\subsection{Middle frequency part }

\begin{lemma}\label{Middle freq lemma} Let $0< \epsilon< \frac{1}{2}$, let $E \subset  B(0,1) \subset \H^n$ have measure $|E| \geq \epsilon$, and let $J=J(\epsilon)$ be a sufficiently large natural number depending on $\epsilon$. Suppose that $0 < r_{J} < \cdots < r_{1} \leq 1 $ are a sequence of scales with $r_{j+1} \leq r_{j}/2$ for all $1 \leq j < J$. Let $r_{j}^2 < \delta \ll 1$. Then for at least one $1 \leq j \leq J$, one has
\begin{equation}\label{middle freq estimate}
    \int_{\R^{*}} \sum_{k \, : \, \delta / r_{j}^{2}  \leq (2k+1) |\lambda |  \leq 1/ (\delta r_{j}^2) } R_{k}(\lambda r^{2},\sigma) \sum_{|\alpha|=k} \langle \hat{f}(\lambda)^{*} \hat{f}(\lambda) \Phi_{\alpha}^{\lambda},\Phi_{\alpha}^{\lambda} \rangle \ |\lambda|^n d\lambda \lesssim \epsilon |E|.
\end{equation}
\end{lemma}
\begin{proof}
We simply observe that since $\textstyle \frac{\Gamma(k+1)}{\Gamma(n) \Gamma(k+n)} \, |\varphi^{n-1}_{k, \lambda}(x)| \leq 1$ so $|R_{k}(\lambda, \sigma)| \lesssim 1$, for all $\lambda \in \mathbb{R}$. Thus, of course, we can estimate left hand side of (\ref{middle freq estimate}) by
\begin{equation*}
    \int_{\R^{*}} \sum_{k \, : \, \delta / r_{j}^{2}  \leq (2k+1) |\lambda |  \leq 1/ (\delta r_{j}^2) } \, \sum_{|\alpha|=k} \langle \hat{f}(\lambda)^{*} \hat{f}(\lambda) \Phi_{\alpha}^{\lambda},\Phi_{\alpha}^{\lambda} \rangle \ |\lambda|^n d\lambda, 
\end{equation*}
and by an application of Plancherel's theorem (\ref{Plancherel}) this in turn is crudely estimated by
\begin{equation}
     \int_{\R^{*}} \sum_{k \geq 0 } \, \sum_{|\alpha|=k} \langle \hat{f}(\lambda)^{*} \hat{f}(\lambda) \Phi_{\alpha}^{\lambda},\Phi_{\alpha}^{\lambda} \rangle \ |\lambda|^n d\lambda = C_n  \int_{\R^{*}} \| \hat{f}(\lambda)\|^{2}_{\text{HS}} \, |\lambda|^n d \lambda = C_n \,  |E|.
\end{equation}
However, because of the lacunarity hypothesis, $r_{j+1} \leq r_{j}/2$, the annulii $A^{\delta}_{j}:= \lbrace (\lambda, \mu ) \in \mathbb{R} \times (2 \mathbb{N} +n) : \delta / r_{j}^2 \leq \lambda \mu \leq 1 /  \delta r_{j}^2 \rbrace$ only overlap with multiplicity $O \left( \log \frac{1}{\delta} \right)$, hence the Plancherel bound, in fact, yields
\begin{equation*}
   \sum_{j=1}^{J}  \int_{\R^{*}} \sum_{k \, : \, \delta / r_{j}^{2}  \leq (2k+1) |\lambda |  \leq 1/ (\delta r_{j}^2) } \, \sum_{|\alpha|=k} \langle \hat{f}(\lambda)^{*} \hat{f}(\lambda) \Phi_{\alpha}^{\lambda},\Phi_{\alpha}^{\lambda} \rangle \ |\lambda|^n d\lambda \lesssim    \log \frac{1}{\delta} \ |E|,
\end{equation*}
and hence by the pigeonhole principle we can choose a scale $r_{j}$, for some $1 \leq j \leq J$, for which
\begin{equation*}
    \int_{\R^{*}} \sum_{k \, : \, \delta / r_{j}^{2}  \leq (2k+1) |\lambda |  \leq 1/ (\delta r_{j}^2) } \, \sum_{|\alpha|=k} \langle \hat{f}(\lambda)^{*} \hat{f}(\lambda) \Phi_{\alpha}^{\lambda},\Phi_{\alpha}^{\lambda} \rangle \ |\lambda|^n d\lambda \lesssim   \frac{ \log \frac{1}{\delta} } {J}  \, |E|.
\end{equation*}
We will now compare it with the contribution of the low frequencies, (\ref{low frequency contri}). Choose $J=J(\delta, \epsilon)$ large enough so that $\textstyle \frac{\log 1/ \delta}{J} \, |E| \lesssim \epsilon |E|$, or equivalently $\frac{1}{ {\displaystyle  \epsilon}} \log 1 / \delta \lesssim J$, that is for $J = J (\delta, \epsilon)$ large enough, $J = O \big( \frac{1}{ {\displaystyle  \epsilon}} \log 1 / \delta \big) $, one can capitalize this ``good" scale, $r_{j}$, to make the contribution due to the medium frequencies small compared to that the low frequencies (which was $|E|^2$), and one can conclude the desired estimate, (\ref{middle freq estimate}).
\end{proof}
\begin{remark}
    We remark that the bound on the number $J$ of scales is  $ J \gtrsim \frac{1}{ {\displaystyle  \epsilon}} \log \frac{1}{ {  \delta}} \gtrsim \frac{1}{ {\displaystyle  \epsilon}} \log \frac{1}{ {\displaystyle  \epsilon}}$. So, we can take $j = O \big( \frac{1}{ {\displaystyle  \epsilon}} \log \frac{1}{ {\displaystyle  \epsilon}} \big)$, which is the same dependence on $\epsilon$ as in the Euclidean case.
\end{remark}

\subsection{High frequency part }
After locating the scale $r_{j}$ which takes care of the contribution (\ref{Medium}) of the medium frequencies as in \Cref{Middle freq lemma}, we proceed to address high frequencies. We use the precise ``decay" of the coefficients $R_{k}(\lambda, \sigma)$, for $ \mu |\lambda|$ large from \Cref{coeff. surface measure} to show that the contribution (\ref{High}) of the high frequencies is negligible.

Using Plancherel's theorem and the \Cref{coeff. surface measure}, one can make contribution (\ref{High}) of the high frequencies small as well:
\begin{align}\label{High Freqncy estimate}
      & \left|  \int_{\R^{*}} \sum_{k \, : \,  2(2k+n) \lambda |  >  1/ (\delta r_{j}^2) } R_{k}(\lambda r_{j}^{2},\sigma) \sum_{|\alpha|=k} \langle \hat{f}(\lambda)^{*} \hat{f}(\lambda) \Phi_{\alpha}^{\lambda},\Phi_{\alpha}^{\lambda} \rangle \ |\lambda|^n d\lambda \right| \notag \\
      & \lesssim  \int_{\R^{*}} \sum_{k \, : \,  2(2k+n) |\lambda |  >  1/ (\delta r_{j}^2) } \left( 2(2k+n)|\lambda| r_{j}^2 \right)^{- \alpha(n)}    \sum_{|\alpha|=k} \langle \hat{f}(\lambda)^{*} \hat{f}(\lambda) \Phi_{\alpha}^{\lambda},\Phi_{\alpha}^{\lambda} \rangle \ |\lambda|^n d\lambda \notag \\
       & \lesssim \delta^{\alpha(n)} \int_{\R^{*}} \sum_{k \, : \,  2(2k+n) |\lambda |  >  1/ (\delta r_{j}^2) }   \,  \sum_{|\alpha|=k} \langle \hat{f}(\lambda)^{*} \hat{f}(\lambda) \Phi_{\alpha}^{\lambda},\Phi_{\alpha}^{\lambda} \rangle  \ |\lambda|^n d\lambda \notag \\
       & \lesssim \delta^{\alpha(n)} \int_{\R^{*}} \sum_{k =0 }^{\infty}   \left|  \sum_{|\alpha|=k} \langle \hat{f}(\lambda)^{*} \hat{f}(\lambda) \Phi_{\alpha}^{\lambda},\Phi_{\alpha}^{\lambda} \rangle \right| \ |\lambda|^n d\lambda \notag \\
       & \simeq \delta^{\alpha(n)} (2 \pi)^{-n-1} \int_{\R^{*}} \| \hat{f}(\lambda) \|^{2}_{\text{HS}}  \ |\lambda|^n d\lambda \notag \\
       & = \delta^{\alpha(n)} \| f \|^{2}_{L^{2}(\H^n)} \notag \\
       & = \delta^{\alpha(n)} |E|, \ \ \ ( \text{recall} \, f = 1_{E}). 
   \end{align}
Altogether, we can now quickly prove the quantitative version of our main theorem, that is \Cref{FKW-quantitative}.  
\subsection{Proof of \Cref{FKW-quantitative}} We continue the discussion before Subsection \ref{Low part}. For $\delta \ll 1$ in view of \Cref{Middle freq lemma} we get a single scale $r_{j}$ such that $r_{j}^2 < \delta$ satisfying (\ref{middle freq estimate}), making the contribution (\ref{Medium}) of medium frequency small compared to that of  the lower bound  (\ref{low frequency contri}) in the low frequencies case. And, for $\delta$ small enough (\ref{High Freqncy estimate}) implies the contribution (\ref{High}) of the high frequencies is negligible as compared to that of the low frequency term, (\ref{low frequency contri}):
\begin{equation*}
    \begin{split}
        & \int_{\R^{*}} \sum_{k=0}^{\infty} R_{k}(\lambda r^{2},\sigma) \sum_{|\alpha|=k} \langle \hat{f}(\lambda)^{*} \hat{f}(\lambda) \Phi_{\alpha}^{\lambda},\Phi_{\alpha}^{\lambda} \rangle \ d\mu(\lambda) \\
        & \geq C |E|^2 - {\textstyle\frac{C}{3} } \, \epsilon |E| - C \delta^{\alpha(n)} |E|\\
        & \geq \frac{C}{3}|E| \epsilon \gtrsim \epsilon^2,
    \end{split}
\end{equation*}
by choosing $\delta \ll 1$ so that $\delta^{\alpha(n)} \leq 3 \epsilon$ since $\alpha(n)= n/2 - 1/4>0$ for all $n \geq 1$.  Thus, (\ref{euivalent main ineq}) follows. This completes the proof of the theorem.

 \section{Proof of Lemma \ref{coeff. surface measure}: \\ ``Decay" of  coefficients in the spectral decomposition of $\widehat{\sigma}$  }\label{mainnovelproof}
Below, we give the ``decay" of $R_{k}(\lambda, \sigma)$ in the high frequency case, namely,  $\mu |\lambda| \gg 1$. First we will prove the lemma for large $k$, whereas small $k$ will be treated towards the end of the proof. We continue to allow the implied constants in the asymptotic notation to depend on $n$, unless specified otherwise. Fix $k_{0}$ large enough, so that we can use the asymptotic for Laguerre polynomials in \Cref{Laguerre-Asymp-Our-setup} for $k \geq k_{0}$.  Recall that 
\begin{equation}\label{Coeff. Rk as Laguerre integral}
   \begin{split}
        R_{k}(\lambda, \sigma) 
        & = c_{n} \frac{\Gamma(k+1)}{\Gamma(k+n)} \int_{- \pi/2}^{\pi/2} L_{k}^{n-1} \left( {\textstyle\frac{1}{2} } |\lambda|  \cos{\theta}  \right) \, e^{-\frac{1}{4} |\lambda| \cos{\theta}}  e^{i \frac{1}{4} \lambda \sin{\theta}} (\cos{\theta})^{n-1} d \theta.
   \end{split}
\end{equation}   
By symmetry, we can assume $\lambda>0$. Also, recall the notation $\mu = 2(2k +n)$ introduced in the notation subsection of the paper; this will be convenient to use in presentation of the argument that follows. To show Lemma \ref{coeff. surface measure} we consider the cases $ \frac{1}{\mu} \leq \lambda \lesssim 1$ and $\lambda \gg 1$, separately.

\noindent \textbf{When $ \frac{1}{\mu} \leq \lambda \lesssim 1$. } Since $\lambda \lesssim 1$, so the the factor $ e^{i \frac{1}{4} \lambda \sin{\theta}}$, being $\sim 1$, produces no  significant extra oscillations. By abuse of notation we continue denoting  integral in (\ref{Coeff. Rk as Laguerre integral})  without $ e^{i \frac{1}{4} \lambda \sin{\theta}}$ by $R_{k}(\lambda, \sigma)$ only:
\begin{equation}\label{lambda less than 1 Coeff. Rk }
   \begin{split}
        R_{k}(\lambda, \sigma) 
        & = c_{n} \frac{\Gamma(k+1)}{\Gamma(k+n)} \int_{- \pi/2}^{\pi/2} L_{k}^{n-1} \left( {\textstyle\frac{1}{2} } \lambda  \cos{\theta}  \right) \, e^{-\frac{1}{4} \lambda \cos{\theta}}  (\cos{\theta})^{n-1} d \theta.
   \end{split}
\end{equation}
After change of variable $x = \frac{1}{2} \lambda \cos{\theta}$, we see that
\begin{equation}\label{x var. lambda less than 1 Coeff. Rk}
   \begin{split}
        R_{k}(\lambda, \sigma)
        & = C_{n} \lambda^{-n +1} \frac{\Gamma(k+1)}{\Gamma(k+n)}  \int_{0}^{ \lambda /2}  L_{k}^{n-1}(x) \, e^{-\frac{1}{2} x }  x^{n-1} \frac{dx}{\sqrt{(\lambda/2)^2 -  x^2}},
   \end{split}
\end{equation}
for some $C_n>0$. Since, $\lambda \mu \gg 1$, we have only one possible situation which is $\frac{1}{\mu} \leq \lambda /2 < \lambda \lesssim 1$. Inspired by the asymptotic of the Laguerre polynomial we decompose the region of integration into two pieces:
\begin{equation}\label{split lambda less than 1 Coeff. Rk }
   \begin{split}
         R_{k}(\lambda, \sigma)
        & = C_{n} \lambda^{-n +1} \frac{\Gamma(k+1)}{\Gamma(k+n)} \left\lbrace  \int_{0}^{ 1 / \mu} + \int_{1/ \mu}^{ \lambda /2}  \right\rbrace L_{k}^{n-1}(x) \, e^{-\frac{1}{2} x }  x^{n-1} \frac{dx}{\sqrt{(\lambda/2)^2 -  x^2}}\\
        & =: J_{1} + J_{2},
   \end{split}
\end{equation}
For the first piece we use the asymptotic (\ref{less than q by mu}) to obtain
\begin{equation*}
   \begin{split}
         |J_{1}|
        & \simeq \lambda^{-n +1}   \int_{0}^{ 1 / \mu}      x^{n-1} \frac{dx}{\sqrt{(\lambda/2)^2 -  x^2}}\\
        & \simeq  \lambda^{-n +1}   \int_{0}^{ 1 / \mu}      x^{n-1} \frac{dx}{\lambda}\\
        & \simeq \lambda^{-n}   \mu^{-n},
   \end{split}
\end{equation*}
where the second equivalence follows from the fact that $\lambda/2 \leq \lambda/2 + x \leq \lambda$, and $\lambda = \frac{1}{\mu} \lambda \mu \simeq \frac{1}{\mu} (\lambda \mu /2 -1) = \lambda/2 - 1 / \mu \leq \lambda /2 - x \leq \lambda /2$ which in turn uses our hypothesis $\lambda \mu \gg 1$. Whereas, using the asymptotic (\ref{1 by mu to mu}) we can bound the second piece as
\begin{equation} \label{J2 estimate}
   \begin{split}
         |J_{2}|
        & \lesssim \lambda^{-n +1}   \int_{ 1 / \mu}^{\lambda/2}      x^{n-1} (x \mu )^{-\frac{2n-1}{4}} \frac{dx}{\sqrt{(\lambda/2)^2 -  x^2}}\\
         & = \lambda^{-n +1} \mu^{- \frac{n}{2} + \frac{1}{4}}   \int_{ 1 / \mu}^{\lambda/2}      x^{\frac{n}{2} - \frac{3}{4}} \frac{dx}{\sqrt{(\lambda/2)^2 -  x^2}}\\
         & \lesssim \lambda^{-n +1} \mu^{- \frac{n}{2} + \frac{1}{4}} \lambda^{-1/2} \lambda^{ \frac{n}{2} - \frac{1}{2}}  \int_{ 1 / \mu}^{\lambda/2}      x^{- \frac{1}{4}} \frac{dx}{\sqrt{\lambda/2 -  x}}\\
         & \overset{ x \rightarrow \frac{\lambda}{2} x }{ \simeq } \lambda^{-n +1} \mu^{- \frac{n}{2} + \frac{1}{4}} \lambda^{-1/2} \lambda^{ \frac{n}{2} - \frac{1}{2}}  \int_{ 2 / \lambda \mu}^{1}   \lambda^{-1/4}   x^{- \frac{1}{4}} \frac{ \lambda dx}{\lambda^{1/2} \sqrt{\lambda/2 -  x}}\\
         & \simeq \lambda^{- \frac{n}{2} + \frac{1}{4}} \mu^{- \frac{n}{2} + \frac{1}{4}}.
   \end{split}
\end{equation}


Thus, we obtain that
\begin{equation*} \label{J2 estimate cont..}
   \begin{split}
         |J_{2}|
         & \lesssim \lambda^{-n +1} \mu^{- \frac{n}{2} + \frac{1}{4}} \lambda^{-1/2} \lambda^{ \frac{n}{2} - \frac{1}{2}} \lambda^{1/4}\\
         & = \lambda^{- \frac{n}{2} + \frac{1}{4}} \mu^{- \frac{n}{2} + \frac{1}{4}}.
   \end{split}
\end{equation*}

Combining the bounds for $J_{1}, J_{2}$ we obtain that when $\frac{1}{\mu} \leq \lambda \lesssim 1$ then
\begin{equation*}
    |R_{k}(\lambda, \sigma)| \simeq (\lambda \mu )^{-n} + (\lambda \mu )^{-\left( \frac{n}{2} - \frac{1}{4}  \right)} \simeq (\lambda \mu )^{-\left( \frac{n}{2} - \frac{1}{4}  \right)}.
\end{equation*}

\noindent \textbf{When $\lambda \gg 1$.}  In view of symmetry, we will abuse the notation and write $R_{k}(\lambda, \sigma)$ for integral in (\ref{Coeff. Rk as Laguerre integral}) but taken over the interval $[0, \pi/2]$ only. Because we have that $\lambda \gg 1$, this will give more possibility of cancellation due to the factor $ e^{i \frac{1}{4} \lambda \sin{\theta}}$ in $R_{k}(\lambda, \sigma)$. To exploit cancellation due to this factor, we use Van der Corput lemma (see \Cref{Van der Corput}) and decomposition away from the stationary points. Since the only stationary point for the phase function $\theta \mapsto \frac{1}{4} \lambda \sin{\theta}$ in $[0, \pi/2]$ is $ \pi/2$, we decompose the integral in (\ref{Coeff. Rk as Laguerre integral}) near and away from $ \pi/2$:
\begin{equation}\label{I plus II Coeff. Rk as Laguerre integral}
   \begin{split}
         R_{k}  (\lambda, \sigma) 
        & = c_{n} \frac{\Gamma(k+1)}{\Gamma(k+n)}  \left\lbrace  \int_{ |\theta - \pi/2| \gtrsim \lambda^{-1/2}} +  \int_{ |\theta - \pi/2| \lesssim \lambda^{-1/2}} \right\rbrace L_{k}^{n-1} \left( {\textstyle\frac{1}{2} } \lambda  \cos{\theta}  \right) \, e^{ - \frac{1}{4} \lambda \cos{\theta}}  e^{i \frac{1}{4} \lambda \sin{\theta}} (\cos{\theta})^{n-1} d \theta \\
        & =:  R_{k}^{I}(\lambda, \sigma) + R_{k}^{II}(\lambda, \sigma)
   \end{split}
\end{equation}
The first integral, where $  |\theta - \pi/2| \gtrsim \lambda^{-1/2} $, has huge oscillations produced by $ e^{i \frac{1}{4} \lambda \sin{\theta}}$. Hence we expect it to be relatively smaller than the second integral. The second integral  forms the base of the oscillatory integral $R_{k}(\lambda, \sigma)$.  We estimate the second term, $R_{k}^{II}(\lambda, \sigma)$, first.   Observe that, while dealing with the term $R_{k}^{II}(\lambda, \sigma)$, the factor $ e^{i \frac{1}{4} \lambda \sin{\theta}}$, being $\sim e^{i \frac{\lambda}{4}}$, produces no significant extra oscillations in the stationary region, $  |\theta - \pi/2| \gtrsim \lambda^{-1/2} $,  other than the fixed one $e^{i \frac{\lambda}{4}}$. By abuse of notation we continue writing for the  second term with $ e^{i \frac{1}{4} \lambda \sin{\theta}}$ replaced by $ e^{i \frac{1}{4} \lambda}$  as $R_{k}^{II}(\lambda, \sigma)$: 

\begin{equation}\label{II Coeff. Rk as Laguerre integral}
   \begin{split}
        R_{k}^{II}(\lambda, \sigma)
        & = c_{n} \frac{\Gamma(k+1)}{\Gamma(k+n)} \,  e^{i \frac{1}{4} \lambda}  \int_{ |\theta - \pi/2| \lesssim \lambda^{-1/2}}  L_{k}^{n-1} \left( {\textstyle\frac{1}{2} } \lambda  \cos{\theta}  \right) \, e^{-\frac{1}{4} \lambda \cos{\theta}}   (\cos{\theta})^{n-1} d \theta.
   \end{split}
\end{equation}
Indeed, in view of the estimate $ \frac{\Gamma(k+1)}{\Gamma(n)\Gamma(k+n)} |L_{k}^{n-1}(x)| e^{-x/2}   \leq 1$, the difference of the original $R_{k}^{II}(\lambda, \sigma)$ and (\ref{II Coeff. Rk as Laguerre integral}) can easily be seen to be $O(\lambda^{-n/2})$, which is acceptable.
\subsection*{Treatment of the base of oscillatory integral: $R_{k}^{II}(\lambda, \sigma)$, $\lambda \gg 1$}

Let us fix a change of variable in above integral:
\begin{equation}\label{x coords}
    x= \frac{1}{2} \, \lambda \cos{\theta}.
\end{equation}
Then we have 
\begin{equation*}
    0 \leq x \lesssim \lambda^{1/2}, \ \lambda \gg 1,
\end{equation*}
whence we get
\begin{equation}\label{x var. II Coeff. Rk as Laguerre integral}
   \begin{split}
        R_{k}^{II}(\lambda, \sigma)
        & = C_{n} \lambda^{-n +1} \frac{\Gamma(k+1)}{\Gamma(k+n)} \,  e^{i \frac{1}{4} \lambda}   \int_{0}^{c \lambda^{1/2}}  L_{k}^{n-1}(x) \, e^{-\frac{1}{2} x }  x^{n-1} \frac{dx}{\sqrt{\lambda^2 - 4 x^2}}\\
   \end{split}
\end{equation}
for some $0<c \ll 1$, $C_n>0$.

To avoid repetition of similar calculations, set $S_{k}(y)$ to be
\begin{equation}\label{factor of Rk}
    \begin{split}
        S_{k}(y)
        & : =  \frac{\Gamma(k+1)}{\Gamma(k+n)}  \int_{0}^{y}  L_{k}^{n-1}(x) \, e^{-\frac{1}{2} x }  x^{n-1} \frac{dx}{\sqrt{\lambda^2 - 4 x^2}},
   \end{split}
\end{equation}
then $R_{k}^{II}(\lambda, \sigma) = C_{n} \lambda^{-n + 1} \,  e^{i \frac{1}{4} \lambda} \, S_{k}(c \lambda^{1/2})$.

Further analysis of $R_{k}^{II}(\lambda, \sigma)$ is dictated by the behavior of the Laguerre polynomial. We distinguish the following three cases:
\begin{equation*}
    \begin{split}
        & \text{A} : 1 \leq c \lambda^{1/2} \leq b \mu  \ \ \ (\text{``Bessel regime"}),\\
        & \text{B} : a \mu  \leq c \lambda^{1/2} \leq \frac{1}{a} \mu, \ \ 0<a <b<1.  \ \ \ (\text{ up to the ``Airy regime"}),\\
        & \text{C} :  c \lambda^{1/2} \gg \mu  \ \ \  (\text{ up to  ``exponential regime"}).\\
    \end{split}
\end{equation*}
\textbf{Case A.} In view of the asymptotic of Laguerre polynomial in this region, we break
\begin{equation*}\label{x var. II Coeff. Rk as Laguerre integral}
   \begin{split}
        S_{k}(c \lambda^{1/2})
        & =  \frac{\Gamma(k+1)}{\Gamma(k+n)}  \int_{0}^{c \lambda^{1/2}}  L_{k}^{n-1}(x) \, e^{-\frac{1}{2} x }  x^{n-1} \frac{dx}{\sqrt{\lambda^2 - 4 x^2}}\\
        & = \int_{0}^{c 1 / \mu} + \int_{1 / \mu}^{c \lambda^{1/2}} =: \ROMAN{I} + \ROMAN{II}.
   \end{split}
\end{equation*}
For $\ROMAN{I}$,
\begin{equation*}
   \begin{split}
        \ROMAN{I}
        & =  \frac{\Gamma(k+1)}{\Gamma(k+n)}  \int_{0}^{ 1 / \mu}  L_{k}^{n-1}(x) \, e^{-\frac{1}{2} x }  x^{n-1} \frac{dx}{\sqrt{\lambda^2 - 4 x^2}}\\
        & \sim  \int_{0}^{ 1 / \mu}   x^{n-1} \left[ 1 + O \left( (x \mu)^{1/2} \right)  \right] \frac{dx}{\sqrt{\lambda^2 - 4 x^2}} \sim  \int_{0}^{ 1 / \mu}   x^{n-1} \frac{dx}{\lambda}\\
        & \sim  \lambda^{-1} \frac{1}{\mu^{n}},
   \end{split}
\end{equation*}
whereas for $\ROMAN{ II }$, one has
\begin{equation*}
   \begin{split}
        \ROMAN{ II }
        & =  \frac{\Gamma(k+1)}{\Gamma(k+n)}  \int_{1 / \mu}^{c \lambda^{1/2}}  L_{k}^{n-1}(x) \, e^{-\frac{1}{2} x }  x^{n-1} \frac{dx}{\sqrt{\lambda^2 - 4 x^2}}\\
        & \sim  \int_{ 1 / \mu}^{c \lambda^{1/2}}   x^{n-1} O \left(  (x \mu )^{-\frac{2n-1}{4}} \right) \frac{dx}{\sqrt{\lambda^2 - 4 x^2}}.
   \end{split}
\end{equation*}
Thus, 
\begin{equation*}
   \begin{split}
        |\ROMAN{II}| 
        & \lesssim   \int_{ 1 / \mu}^{c \lambda^{1/2}}   x^{n-1} \, (x \mu )^{-\frac{2n-1}{4}} \frac{dx}{\lambda} \simeq \frac{1}{\lambda}   \mu^{-n/2+ 1/4} (\lambda^{1/2})^{n/2 + 1/4}.
   \end{split}
\end{equation*}
Altogether,
\begin{equation}\label{upto mu base}
   \begin{split}
        |S_{k}(c \lambda^{1/2})|
        & \lesssim    \frac{1}{\lambda} \frac{1}{\mu^{n/2 - 1/4}}  \left[  \frac{1}{\mu^{n/2 + 1/4}} +   (\lambda^{1/2})^{n/2 + 1/4} \right]\\
        & \simeq  \frac{1}{\lambda} \frac{1}{\mu^{n/2 - 1/4}}   (\lambda^{1/2})^{n/2 + 1/4}.
   \end{split}
\end{equation}
Therefore,
\begin{equation}\label{Case A estimate}
   \begin{split}
        | R_{k}^{II}(\lambda, \sigma)|
        & \simeq \lambda^{-n+1} \frac{1}{\lambda} \frac{1}{\mu^{n/2 - 1/4}}   (\lambda^{1/2})^{n/2 + 1/4}\\
        & = \frac{1}{(\lambda \mu)^{n/2 - 1/4}} \, \frac{1}{(\lambda )^{n/4 + 1/8}}. 
   \end{split}
\end{equation}

\textbf{Case B.} If $a \mu \leq y \leq \mu - c \mu^{1/3}$ then
\begin{equation*}
    S_{k}(y) = S_{k}(a \mu) +    \frac{\Gamma(k+1)}{\Gamma(k+n)}  \int_{a \mu}^{y}  L_{k}^{n-1}(x) \, e^{-\frac{1}{2} x }  x^{n-1} \frac{dx}{\sqrt{\lambda^2 - 4 x^2}}.
\end{equation*}
From (\ref{upto mu base}), $S_{k}(a \mu) \simeq  \lambda^{-1} \mu^{1/2} $. For the remaining piece we use asymptotic (\ref{Bessel-Airy interface}) to obtain
\begin{equation}\label{B case before nbd of mu}
    \begin{split}
        & \left|   \frac{\Gamma(k+1)}{\Gamma(k+n)}  \int_{a \mu}^{y}  L_{k}^{n-1}(x) \, e^{-\frac{1}{2} x }  x^{n-1} \frac{dx}{\sqrt{\lambda^2 - 4 x^2}}    \right| \\
        & \simeq    \int_{a \mu}^{y}  \mu^{- \frac{4}{3}} \mu^{\frac{1}{12}} (\mu - x)^{- 1/4} \frac{dx}{\sqrt{\lambda^2 - 4 x^2}}\\
        & \simeq  \lambda^{-1} \mu^{- \frac{4}{3}} \mu^{\frac{1}{12}}  \int_{a \mu}^{y}    \frac{dx}{(\mu - x)^{1/4}}\\
         & \overset{ x \rightarrow \mu x}{=} \lambda^{-1} \mu^{- \frac{4}{3}} \mu^{\frac{1}{12}} \mu^{3/4}  \int_{a }^{y \mu^{-1}}    \frac{dx}{(1 - x)^{1/4}}\\
         & \simeq \lambda^{-1} \mu^{- \frac{1}{2}}.
    \end{split}
\end{equation}
If $\mu - c \mu^{1/3} \leq y \leq \mu + c \mu^{1/3}$, $0<c \ll 1$, then
\begin{equation*}
    S_{k}(y) = S_{k}(\mu - c \mu^{1/3}) +    \frac{\Gamma(k+1)}{\Gamma(k+n)}  \int_{\mu - c \mu^{1/3}}^{y}  L_{k}^{n-1}(x) \, e^{-\frac{1}{2} x }  x^{n-1} \frac{dx}{\sqrt{\lambda^2 - 4 x^2}},
\end{equation*}
so using the the Airy regime asymptotic of Laguerre polynomial, the contribution due to the  second term above is  
\begin{equation}
    \begin{split}
        & \simeq   \int_{\mu - c \mu^{1/3}}^{y}   \mu^{- \frac{4}{3}}  \frac{dx}{\lambda} \simeq \lambda^{-1} \mu^{-4/3} \mu^{1/3}= \lambda^{-1} \mu^{-1},
    \end{split}
\end{equation}
which combined with previous estimates give 
\begin{equation}\label{B case around nbd of mu}
    \begin{split}
        |S_{k}(y)| & \simeq  \lambda^{-1} \mu^{1/2} + \lambda^{-1} \mu^{- \frac{1}{2}} +  \lambda^{-1} \mu^{-1}\\
        & \simeq  \lambda^{-1} \mu^{1/2}.
    \end{split}
\end{equation}

Finally, if $\mu + c \mu^{1/3} \leq y \leq \frac{3}{2} \mu$, then we write
\begin{equation*}
    S_{k}(y) = S_{k}(\mu + c \mu^{1/3}) +    \frac{\Gamma(k+1)}{\Gamma(k+n)}  \int_{\mu + c \mu^{1/3}}^{y}  L_{k}^{n-1}(x) \, e^{-\frac{1}{2} x }  x^{n-1} \frac{dx}{\sqrt{\lambda^2 - 4 x^2}},
\end{equation*}
wherein for the second part we can use asymptotic (\ref{Airy merging with exponential}) and thus, in modulus,  it is 
\begin{equation*}
    \begin{split}
        & \lesssim \mu^{-\frac{4}{3}} \mu^{\frac{1}{12}} \int_{\mu + c \mu^{1/3}}^{y} (x-\mu)^{-\frac{1}{4}} e^{\frac{2}{3}  \mu^{-1/2} (x-\mu)^{3/2}  } \frac{dx}{\lambda}\\
        & \overset{x \rightarrow \mu x}{=} \lambda^{-1} \mu^{-\frac{4}{3}} \mu^{\frac{1}{12}} \mu^{-\frac{1}{4}} \int_{1 + c \mu^{-2/3}}^{y/ \mu} (x-1)^{-\frac{1}{4}} e^{\frac{2}{3}  \mu (x-1)^{3/2}  } dx.
    \end{split}
\end{equation*}
A short calculation yields that the last integral is $\simeq \mu^{-1/2}$ so the second part is bounded by $\lambda^{-1} \mu^{-\frac{5}{3}}$. On the other hand bound for the first part is given by (\ref{B case around nbd of mu}). So, in the region $\mu + c \mu^{1/3} \leq y \leq \frac{3}{2} \mu$, we obtain total bound of 
\begin{equation}\label{B case after nbd of mu}
    \begin{split}
        |S_{k}(y)| & \simeq  \lambda^{-1} \mu^{1/2} + \lambda^{-1} \mu^{- \frac{5}{3}} \simeq  \lambda^{-1} \mu^{1/2}.
    \end{split}
\end{equation}

Putting all this together, and using $\lambda^{1/2} \simeq \mu$ in this Case B, we obtain here that 
\begin{equation}
    \begin{split}
        |R_{k}^{II}(\lambda, \sigma)| & \simeq \lambda^{-n + 1}   \lambda^{-1} \mu^{1/2}  
        \simeq  (\lambda \mu )^{-\left( \frac{2n}{3} - \frac{1}{6}  \right)}.
    \end{split}
\end{equation}
\textbf{Case C.} In view of Case B, it suffices to bound the portion of  $S_{k}(y)$ after Airy regime, that is to say $\mu \ll x \leq y=c \lambda^{1/2} $:
\begin{equation*}
   \left|  \frac{\Gamma(k+1)}{\Gamma(k+n)}  \int_{ \frac{1}{a}\mu }^{y}  L_{k}^{n-1}(x) \, e^{-\frac{1}{2} x }  x^{n-1} \frac{dx}{\sqrt{\lambda^2 - 4 x^2}} \right| \ll \lambda^{-1} \mu^{- \frac{n}{2} - \frac{1}{2} - N },
\end{equation*}
for all $N>0$. Proving this estimate would ensure in Case C also we have
\begin{equation}
    \begin{split}
        |R_{k}^{II}(\lambda, \sigma)| & \simeq \lambda^{-n + 1}   \lambda^{-1} \mu^{1/2}  \simeq  (\lambda \mu )^{-\left( \frac{2n}{3} - \frac{1}{6}  \right)}.
    \end{split}
\end{equation}

But from the asymptotic (\ref{expo regime}) of $L_{k}^{n-1}$ in this region the above integral expression is majorized by
\begin{equation*}
     \lambda^{-1} \mu^{-\frac{n}{2}-\frac{1}{2} } \int_{ \frac{1}{a}\mu }^{c \lambda^{1/2}}  x^{ \frac{n}{2} -1 } e^{- \frac{2}{3} x} dx =  \lambda^{-1} \mu^{-\frac{n}{2}-\frac{1}{2} } \left[ \Gamma( {\textstyle\frac{n}{2} } ) - O (\lambda^{-N/2}) - \left( \Gamma( {\textstyle\frac{n}{2} } ) - O (\mu^{-N}) \right) \right]
\end{equation*}
whence the claim follows.

Putting all the cases \textbf{A}, \textbf{B} and \textbf{C} together, we conclude that the second term in (\ref{I plus II Coeff. Rk as Laguerre integral}) has the bound
\begin{equation}
    |R_{k}^{II}(\lambda, \sigma)| \simeq (\lambda \mu )^{-\left( \frac{n}{2} - \frac{1}{4}  \right)}.
\end{equation}

\subsection*{Treatment of the auxiliary part of oscillatory integral: $R_{k}^{I}(\lambda, \sigma)$, $\lambda \gg 1$} Next, we show that the first part, that is $R_{k}^{I}(\lambda, \sigma)$ in (\ref{I plus II Coeff. Rk as Laguerre integral}), is relatively smaller than the second. Here, extra  oscillations given by $e^{i \frac{1}{4} \lambda \sin{\theta}}$ will play the crucial role and will be exploited using the following version (to avoid differentiating Laguerre polynomials) of van der Corput lemma about oscillatory integral, for example see  \cite[pp. 332-334]{Ste93}.

\begin{lemma}\label{Van der Corput}
    Let $\varPhi$ be a smooth function on an interval $I=[a, b]$ and $\varPsi$ be a smooth positive decreasing function on $[a, b]$. Suppose that $|\varPhi'| \geq L$ on $I$ and suppose additionally that $\varPhi'$ is monotonic on $I$. Then, there exist a universal constant $C>0$ such that
    \begin{equation*}
        \left|  \int_{a}^{b} e^{i \varPhi(s)} \varPsi(s)  ds \right| \leq C \frac{\varPsi(a)}{ L}.
    \end{equation*}
In case when amplitude function $\varPsi$ is instead increasing then the above oscillatory integral is bounded by $ C \frac{2 \varPsi(b) - \varPsi(a)}{ L}.$
\end{lemma}

This time we will work on the original coordinate $\theta$ when integrating to avoid working with singular amplitudes:
\begin{equation}\label{Recall 1 II Coeff. Rk as Laguerre integral}
   \begin{split}
        R_{k}^{I}(\lambda, \sigma) 
        & = c_{n} \frac{\Gamma(k+1)}{\Gamma(k+n)}    \int_{0}^{  \pi/2 - c \lambda^{-1/2} } L_{k}^{n-1} \left( {\textstyle\frac{1}{2} } \lambda  \cos{\theta}  \right) \, e^{ - \frac{1}{4} \lambda \cos{\theta}}  e^{i \frac{1}{4} \lambda \sin{\theta}} (\cos{\theta})^{n-1} d \theta.
   \end{split}
\end{equation}
However, to decide the case study and compact presentation, it will be helpful to read integrand in the notation $x=x(\theta)= \frac{1}{2} \, \lambda \cos{\theta}$  as
\begin{equation*}
   c_{n} \lambda^{-n+1} \frac{\Gamma(k+1)}{\Gamma(k+n)} L_{k}^{n-1}(x) \, e^{-\frac{1}{2} x }  x^{n-1} e^{i \frac{1}{4} \lambda \sin{\theta}}
\end{equation*}

In view of the behavior of Laguerre polynomials we need to consider the following three exhaustive cases:
\begin{equation*}
    \begin{split}
        &  \mathcal{A} : 1 \leq  c \lambda^{1/2}  \leq b \mu  \ \ \ (\text{``Bessel regime" and after} ), \\
        & \mathcal{B}: a \mu  \leq c \lambda^{1/2} \leq {\textstyle\frac{1}{a} } \mu, \ \ 0<a <b<1,  \ \ \ (\text{ ``Airy regime" and after}),\\
        & \mathcal{C}: {\textstyle\frac{1}{a} } \mu \leq  c \lambda^{1/2}  \ \ \  (\text{ only  ``exponential regime"}).\\
    \end{split}
\end{equation*}
\noindent \textbf{Case $\mathcal{C}$.} We begin with the case when oscillatory integral lives only  inside the exponential regime. From the asymptotic (\ref{expo regime}) of $L_{k}^{n-1}$ in this region, it suffices to bound 
\begin{equation*}
   \left| \lambda^{-n+1}  \int_{0 \leq \theta \lesssim \frac{\pi}{2} - c \lambda^{-1/2}} e^{i \frac{1}{4} \lambda \sin{\theta}}  \mu^{-\frac{n}{2} - \frac{1}{2}} x^{\frac{n}{2} - 1} \,   e^{-\frac{2}{3} x }  d \theta \right|
\end{equation*}
Here, $x= x(\theta)=\frac{1}{2} \, \lambda \cos{\theta}$. 

Since $\lambda \gg 1$, so for all $n \geq 1$ the amplitude (off the factors constants for integration) $\theta \mapsto x(\theta)^{\frac{n}{2}-1} e^{-\frac{2}{3} x(\theta)}$ increases on $[0, \frac{\pi}{2}- c \lambda^{-1/2}]$. Also, on this interval phase function $\varPhi(\theta):= \frac{1}{4} \lambda \sin{\theta}$ has its derivative bounded below as : $|\varPhi'(\theta)| \gtrsim \lambda^{1/2} $.  Thus, using van der Corput lemma, the above integral is 
\begin{align}
    & \lesssim \lambda^{-n+1} \left( \mu^{-\frac{n}{2} - \frac{1}{2}} \frac{  \left( \lambda^{1/2} \right)^{\frac{n}{2}-1} e^{-\frac{2}{3} \lambda^{1/2}}         }{\lambda^{1/2}} \right) \label{Corput exponential}  \\
        & = (\lambda \mu)^{- \frac{n}{2}+ \frac{1}{4}} \, \mu^{-3/4} \lambda^{- (n+1)/4} e^{-\frac{2}{3} \lambda^{1/2}}, \nonumber
\end{align}
which is acceptable. 

\noindent \textbf{Case $\mathcal{B}$.} We further distinguish several possibilities. \medskip
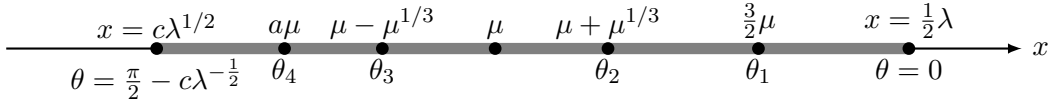
\begin{center}
\begin{tikzpicture}
\draw[thick, -latex] (-1, 0) -- (12.5, 0) node[right] {$x$};

\draw[line width=1.5mm, gray] (1, 0) -- (11, 0);


\foreach \x/\aboveLabel/\belowLabel in {1/$x=c \lambda^{1/2}$/$\theta = \frac{\pi}{2} - c \lambda^{-\frac{1}{2}}$, 2.7/$a \mu$/$\theta_{4}$, 4/$\mu - \mu^{1/3}$/$\theta_{3}$, 5.5/$\mu$/, 7/$\mu + \mu^{1/3}$/$\theta_{2}$,  9/$\frac{3}{2} \mu$/$\theta_{1}$, 11/$x=\frac{1}{2} \lambda$/$\theta=0$}
{
    \filldraw[thick] (\x, 0) circle (2pt);

    \node[above] at (\x, 0) {\aboveLabel};

    \node[below] at (\x, 0) {\belowLabel};
}

\end{tikzpicture}
\end{center}
\medskip









\noindent \textbf{Case $\mathcal{B}_3$.} If $\mu + \mu^{1/3} \leq c \lambda^{1/2} \leq \frac{3}{2} \mu$ then, setting $\frac{\lambda}{2} \cos{ \theta_{1} }:= \frac{3}{2} \mu $, we split the integral in $R_{k}^{I}(\lambda, \sigma)$ into two integrals $\textbf{I}$ and $\textbf{II}$ taken over intervals $[0, \theta_{1}]$ and $[\theta_{1}, \pi/2 - c \lambda^{-1/2}]$, respectively. For $\textbf{II}$, since $\mu + \mu^{1/3} \leq c \lambda^{1/2} \leq x(\theta) \leq \frac{3}{2} \mu$ in this case, we use the asymptotic (\ref{Airy merging with exponential}) to bound
\begin{equation*}
  |\textbf{II}| \simeq \left| \lambda^{-n+1}  \int_{\theta_{1} \leq \theta \lesssim \frac{\pi}{2} - c \lambda^{-1/2}} \mu^{{-\textstyle\frac{4}{3}} } \,   \mu^{ {\textstyle\frac{1}{12} } }  (x- \mu)^{ - {\textstyle\frac{1}{4}  } } \, e^{ - {\frac{2}{3}  } \mu^{- \frac{1}{2}} (x- \mu)^{3/2}    }  \, e^{i \frac{1}{4} \lambda \sin{\theta}}    d \theta \right|.
\end{equation*}
In this interval $[\theta_{1}, \pi/2 - c \lambda^{-1/2}]$  derivative of the phase $\varPhi'(\theta)=x(\theta)= \frac{1}{2} \lambda \cos{\theta}$ is still bounded below by $\lambda^{1/2}$ and amplitude is clearly increasing in $\theta$ variable. So from van der Corput lemma 
\begin{equation}\label{theta1 to theta2}
  \begin{split}
   \lambda^{n-1}   | \textbf{II} | & \lesssim    \, \frac{ \mu^{{-\textstyle\frac{4}{3}} } \,   \mu^{ {\textstyle\frac{1}{12} } }  (c \lambda^{1/2} - \mu)^{ - {\textstyle\frac{1}{4}  } } \, e^{ - {\frac{2}{3}  } \mu^{- \frac{1}{2}} ( c \lambda^{1/2} - \mu  )^{3/2}    } } { \lambda^{1/2} } \\
   & \lesssim    \, \mu^{{-\textstyle\frac{4}{3}} } \,      \lambda^{-1/2}, 
  \end{split}
\end{equation}
or equivalently, 
\begin{equation*}
  \begin{split}
      | \textbf{II} | & \lesssim     (\lambda \mu)^{ -\frac{n}{2} + 1/4} \mu^{-\frac{n}{2}- 1- \frac{1}{12}}  .
  \end{split}
\end{equation*}
Whereas integral $\textbf{I}$, which is taken over $[0, \theta_{1}]$ (or equivalently in $x$ coordinates it is the interval $[\frac{3}{2} \mu, \frac{1}{2} \lambda]$) is estimated as in the  Case $\mathcal{C}$ (this time amplitude maximizes at $x(\theta_{1})= \frac{3}{2} \mu$, and lower bound on phase's derivative, which is  $x(\theta)$, is again $\frac{3}{2} \mu$): 
\begin{equation*}
  \begin{split}
      |\textbf{I}| & \lesssim \lambda^{-n+1} \mu^{-\frac{n}{2} - \frac{1}{2}} \frac{  \left( \frac{3}{2} \mu \right)^{\frac{n}{2}-1} e^{-\frac{2}{3} \frac{3}{2} \mu }         }{{\textstyle\frac{3}{2}} \mu} \lesssim  (\lambda \mu)^{ -\frac{n}{2} + 1/4} \mu^{-\frac{n}{2}-\frac{5}{4}} e^{-c \mu} .
  \end{split}
\end{equation*}
Therefore, 
\begin{equation*}
  \begin{split}
      |R_{k}^{I}(\lambda, \sigma)| \lesssim  (\lambda \mu)^{-\frac{n}{2} + \frac{1}{4}}  \mu^{-\frac{n}{2} -1-  \frac{1}{4}} e^{-c \mu}.
  \end{split}
\end{equation*}

\noindent \textbf{Case $\mathcal{B}_2$.} If $\mu - \mu^{1/3} \leq c \lambda^{1/2} \leq \mu + \mu^{1/3}$. Set $\frac{\lambda}{2} \cos{\theta_{2}}:=\mu + \mu^{1/3}$. Now, we break the the integral in $R_{k}^{I}(\lambda, \sigma)$ into three pieces $\textbf{J}_1$, $\textbf{J}_2$, $\textbf{J}_{3}$ taken over the regions $[0, \theta_{1}]$, $[\theta_{1}, \theta_{2}]$ and $[\theta_{2}, \pi/2 - c \lambda^{-1/2}]$, respectively. For $\textbf{J}_{3}$, observe that the interval $[\theta_{2}, \pi/2 - c \lambda^{-1/2}]$ in $x=x(\theta)$ coordinates reads as $\mu - \mu^{1/3} \leq c \lambda^{1/2} \leq x \leq \mu + \mu^{1/3}$. Using the Airy region asymptotic  (\ref{around mu}) and then van der Corput lemma we obtain
\begin{equation*}
  \begin{split}
      |\textbf{J}_{3} | &  \simeq \left| \lambda^{-n+1}  \int_{\theta_{2} \leq \theta \lesssim \frac{\pi}{2} - c \lambda^{-1/2}} \mu^{{-\textstyle\frac{4}{3}} }   \, e^{i \frac{1}{4} \lambda \sin{\theta}}    d \theta \right|\\
      & \lesssim \lambda^{-n+1} \frac{  \mu^{{-\textstyle\frac{4}{3}} }  }{ \lambda^{1/2}   } \simeq (\lambda \mu)^{- \frac{n}{2} + \frac{1}{4}} \cdot \mu^{- \frac{n}{2} - 1 - \frac{1}{12}} .
  \end{split}
\end{equation*}
The prospective computations for the term $\textbf{J}_{2}$ (that is using van der Corput on the interval $[\theta_{1}, \theta_{2}]$, the relation between $\theta$ and $x=x(\theta)$, and the fact that here $\lambda^{1/2} \simeq \mu$) when compared with the corresponding estimation in (\ref{theta1 to theta2}) suggest that  the contribution due to $\textbf{J}_{2}$ is sames as that of $\textbf{II}$:
\begin{equation*}
  \begin{split}
     | \textbf{J}_{2} | & \lesssim   \lambda^{ -\frac{n}{2} + \frac{1}{4}} \mu^{-n-1} e^{-c \mu} \simeq (\lambda \mu)^{ -\frac{n}{2} + 1/4} \mu^{-\frac{n}{2}-1-\frac{1}{12}}  .
  \end{split}
\end{equation*}
And, with the same argument as in Case $\mathcal{C}$ (with $c\lambda^{1/2}$ replaced by $\frac{3}{2} \mu$ inside the parentheses in (\ref{Corput exponential})) we also get
\begin{equation*}
  \begin{split}
     | \textbf{J}_{1} | & \lesssim   \lambda^{ -\frac{n}{2} + \frac{1}{4}} \mu^{-n-1} e^{-c \mu} \simeq (\lambda \mu)^{ -\frac{n}{2} + 1/4} \mu^{-\frac{n}{2}-1-\frac{1}{4}} e^{-c \mu} .\end{split}
\end{equation*}
Therefore, altogether, we get  
\begin{equation} \label{conclusion CaseB2}
  \begin{split}
      |R_{k}^{I}(\lambda, \sigma)| \lesssim  (\lambda \mu)^{- \frac{n}{2} + \frac{1}{4}} \cdot \mu^{- \frac{n}{2} - 1 - \frac{1}{12}}.
  \end{split}
\end{equation}

\noindent \textbf{Case $\mathcal{B}_1$.} If $a \mu  \leq c \lambda^{1/2} \leq \mu - \mu^{1/3}$. In view of the case study Case $\mathcal{B}_{2}$, we only need to concentrate on the interval $[\theta_{3}, \pi / 2 - c \lambda^{-1/2}]$ where $\frac{\lambda}{2} \cos{\theta_{3}}:= x(\theta_{3})=\mu - \mu^{1/3}$. The total contribution due to remaining intervals $[0,\theta_{1}]$, $[\theta_{1},\theta_{2}]$, $[\theta_{2},\theta_{3}]$ is given by (\ref{conclusion CaseB2}). Let $ \textbf{K}_{4}$ denote the portion of $R_{k}^{I}(\lambda, \sigma)$ considered on $[\theta_{3}, \pi / 2 - c \lambda^{-1/2}]$. Then, in view of (\ref{Bessel-Airy interface}), we have
\begin{equation}\label{1 last piece Airy-Bessel}
  |\textbf{K}_{4}| \simeq \left| \lambda^{-n+1}  \int_{\theta_{3} \leq \theta \lesssim \frac{\pi}{2} - c \lambda^{-1/2}} \mu^{{-\textstyle\frac{4}{3}} } \,   \mu^{ {\textstyle\frac{1}{12} } }  ( \mu - x  )^{ - {\textstyle\frac{1}{4}  } }   \, e^{i \left( \pm \frac{2}{3}  |\xi(x(\theta))|^{3/2}    + \frac{1}{4} \lambda \sin{\theta}  \right) }    d \theta \right|.
\end{equation}
 One can easily lower bound the derivative of the phase $\Phi_{\pm}(\theta):= \pm \frac{2}{3}  |\xi(x(\theta))|^{3/2}    + \frac{1}{4} \lambda \sin{\theta}$  to be
\begin{equation*}
    \left| \frac{d}{d \theta} \Phi_{\pm} (\theta) \right| \gtrsim \mu^{5/3},
\end{equation*}   
for $\theta_{3} \leq \theta \lesssim \frac{\pi}{2} - c \lambda^{-1/2}$. Indeed, since we are in the case $2 (a)$ of \Cref{Laguerre-Asymp-Our-setup}, so $|\xi( x(\theta)) | =  \mu^{2/3} \phi \left( {\textstyle\frac{x(\theta)}{\mu} } \right)$ or equivalently $|\xi( x(\theta)) |^{3/2} =  \mu \, \phi \left( {\textstyle\frac{x(\theta)}{\mu} } \right)^{3/2}$. From \cite[page 698]{Askey-Wainger}, $\phi(t)$ is decreasing on the interval $(0,1]$, and $\phi(t) \left[ \phi^{\prime}(t) \right]^2  = \frac{1}{4} \left( \frac{1}{t} -1 \right)$ for all $0 < t< \infty$ whence $ \frac{d}{d \theta} \left( |\xi( x(\theta)) |^{3/2} \right) $ $ = { \textstyle\frac{3}{8} } \lambda \left(  \frac{2 \mu}{ \lambda \cos{\theta}} -1 \right)^{1/2} \sin{\theta}$. Therefore, 
\begin{equation*}
    \frac{d}{d \theta} \left( \Phi_{\pm} (\theta)  \right)=  \frac{1}{4} \lambda  \left[ \pm \left(  \frac{2 \mu}{ \lambda \cos{\theta}} -1 \right)^{1/2} \sin{\theta} + \cos{ \theta} \right].
\end{equation*}
On the one hand, since $\theta_{3} \leq \theta \lesssim \frac{\pi}{2} - c \lambda^{-1/2}$ or equivalently $a \mu \leq c \lambda^{1/2} \leq x (\theta) \leq \mu - \mu^{1/3}$, one can lower bound ${\textstyle\frac{\lambda}{4}} \sin{\theta} \left( \frac{2 \mu}{\lambda \cos{\theta}} -1 \right)^{1/2} \gtrsim \lambda \mu^{-1/3} \simeq \mu^{5/3}$ and on the other hand we have bound $\frac{\lambda}{4} \cos{\theta}= x(\theta) \leq \mu - \mu^{1/3}$. The claimed lower bound on the derivative of $\Phi_{\pm}$ now follows from the triangle inequality. As a consequence, van der Corput lemma and the observation that in  $\theta$-coordinates the ``amplitude", here, $ \theta \mapsto (\mu - x(\theta) )^{-1/4}$ is decreasing, yield the bound
\begin{equation*}
  \begin{split}
 |\textbf{K}_{4}| & \lesssim \lambda^{-n+1} \mu^{-\frac{4}{3} } \mu^{- \frac{1}{12}} \frac{  \left(  \mu - x(\theta_{3}) \right)^{-1/4}  }{ \mu^{5/3} }, \ \ \  (\, x(\theta_{3})= \mu - \mu^{1/3} \,), \\
 & = \lambda^{-n+1} \mu^{-3} \\
  & \simeq (\lambda \mu )^{-\frac{n}{2} + \frac{1}{4} } \mu^{-\frac{n}{2} - 1 - \frac{3}{4}},
  \end{split}
\end{equation*}
which is acceptable.

\noindent \textbf{Case $\mathcal{A}$.} Within this case of $1 \leq  c \lambda^{1/2}  \leq b \mu$ the only new situation which can occur is that of when $1 \leq  c \lambda^{1/2} \leq \lambda \leq b \mu$. The remaining has already been covered in previous cases $\mathcal{B}_{3},\mathcal{B}_{2},\mathcal{B}_{1}, \mathcal{C}$. In this situation of $1 \leq  c \lambda^{1/2} \leq \lambda \leq b \mu$, we can estimate full piece $R_{k}^{I}(\lambda, \sigma)$ once because we are in Bessel regime $1 / \mu \leq 1 \leq x  \leq \frac{1}{2} \lambda \leq b \mu$. In view of the asymptotic (\ref{1 by mu to mu}), it suffices to estimate
\begin{equation}\label{Purely Bessel}
   \left| \lambda^{-n+1}  \int_{0 \leq \theta \lesssim \frac{\pi}{2} - c \lambda^{-1/2}} \,   x^{\frac{n}{2}-1} \mu^{-\frac{n}{2}+\frac{1}{2}} \, \psi^{\prime}({\textstyle\frac{x}{\mu}})^{-1/2}   \, e^{i \left( \pm \mu   \psi \left({\textstyle\frac{x}{\mu}} \right)   + \frac{1}{4} \lambda \sin{\theta}  \right) }    d \theta \right|.
\end{equation}
First we study where, in the region of integration, the phase $\Psi_{\pm}:= \pm \mu   \psi \left({\textstyle\frac{x}{\mu}} \right)   + \frac{1}{4} \lambda \sin{\theta}$ is stationary. Recall from \cite[page 697]{Askey-Wainger} that $\psi$ satisfies $ \psi^{\prime}(t)= {\textstyle\frac{1}{2}}\left(  \frac{1}{t} -1 \right)^{1/2}$, $0<t<\infty$, whence
\begin{equation}\label{stationary2}
    \frac{d}{d \theta} \left( \Psi_{\pm} (\theta)  \right)=  \frac{1}{4} \lambda  \left[ \mp \left(  \frac{2 \mu}{ \lambda \cos{\theta}} -1 \right)^{1/2} \sin{\theta} + \cos{ \theta} \right].
\end{equation}
Clearly, the derivative of the phase $\Psi_{-}$, in (\ref{Purely Bessel}), has a lower bound of $c \lambda^{1/2}$ throughout the region of integration, while the amplitude function decreases in $\theta$. An application of the van der Corput lemma then easily yields the bound of $(\lambda \mu)^{-\frac{n}{2} + \frac{1}{4}} \lambda^{-\frac{1}{2}}$ for (\ref{Purely Bessel}), which is acceptable for our purpose. Whereas, for $\Psi_{+}$ let $\theta_{0}$ be its stationary point, that is, $  \Psi_{+}^{\prime}  (\theta_{0} )=0$. Rearranging (\ref{stationary2}) at $\theta_{0}$ we get $\cos^{2}{\theta_{0}} + \frac{\lambda}{2 \mu} \cos{\theta_{0}} -1 =0$ or equivalently 
 \begin{equation*}
     \cos{\theta_{0}}= - \frac{\lambda}{4 \mu} + \sqrt{ \left( \frac{\lambda}{4 \mu}  \right)^2 + 1 }.
 \end{equation*}
 Since $\frac{\lambda}{2 \mu} \leq b/2 \ll 1$ so $\cos{\theta_{0}} \sim 1 - \frac{\lambda}{4 \mu}$  which forces the stationary point $\theta_{0}$ to lie inside the region of integration $[0,  \frac{\pi}{2} - c \lambda^{-1/2}] $. 
Moreover, using this approximation of $\cos{\theta_{0}}$,  we can compute the $\Psi_{+}^{\prime \prime}(\theta_{0})$ to get an idea of the size of stationary region:
 \begin{equation}\label{size stationary2}
   \begin{split}
        -\frac{d^2}{d \theta^2} \left( \Psi_{+} (\theta_{0})  \right) & =  \frac{1}{2} \lambda  \left[  \left(  \frac{2 \mu}{ \lambda \cos{\theta_{0}}} -1 \right)^{-1/2} \left\lbrace  \frac{1}{2 \cos{\theta_{0}}}    + \frac{2\mu}{\lambda} - \cos{ \theta_{0}}  \right\rbrace + \sin{\theta_{0}}  \right]\\
        & \simeq \frac{1}{2} \lambda  \left(  \frac{\lambda}{ 2 \mu } \cos{\theta_{0}}  \right)^{1/2} \left[   \frac{1}{2 \cos{\theta_{0}}}    + \frac{2\mu}{\lambda} - \cos{ \theta_{0}}     + 1  \right]\\
        & \simeq   \frac{1}{2\sqrt{2}} \frac{ \lambda^{3/2} }{\mu^{1/2}}  \left(  1-  \frac{\lambda}{ 4 \mu }   \right)^{1/2} \left[   \frac{1}{2 }\left(  1-  \frac{\lambda}{ 4 \mu }   \right)^{-1}    + \frac{2\mu}{\lambda} + \frac{\lambda}{2 \mu}  \right]\\
        & \simeq   \frac{1}{2\sqrt{2}} \frac{ \lambda^{3/2} }{\mu^{1/2}}  \left(  1-  \frac{\lambda}{ 8 \mu }   \right) \left[   \frac{1}{2 }\left(  1 + \frac{\lambda}{ 4 \mu }   \right)    + \frac{2\mu}{\lambda} + \frac{\lambda}{2 \mu}  \right]\\
        & \simeq   \frac{ \lambda^{3/2} }{\mu^{1/2}}  \left(  1-  \frac{\lambda}{ 8 \mu }   \right) \frac{\mu}{\lambda} \simeq   \lambda^{1/2} \mu^{1/2}.
   \end{split}
\end{equation}
Whence, the main contribution to (\ref{Purely Bessel}) is given by the stationary region
\begin{equation*}\label{Purely Bessel1}
   \begin{split}
      &  \left| \lambda^{-n+1}  \int_{ | \theta - \theta_{0} | \lesssim  (\lambda \mu)^{-1/4}} \,   x^{\frac{n}{2}-1} \mu^{-\frac{n}{2}+\frac{1}{2}} \, \psi^{\prime}({\textstyle\frac{x}{\mu}})^{-1/2}      d \theta \right| \\
      & \simeq \left| \lambda^{-n+1}  \int_{ | \theta - \theta_{0} | \lesssim  (\lambda \mu)^{-1/4}} \,   x^{\frac{n}{2}-1} \mu^{-\frac{n}{2}+\frac{1}{2}} \, ({\textstyle\frac{x}{\mu}})^{1/4}      d \theta \right|\\
      & \simeq  \lambda^{-n+1} \lambda^{\frac{n}{2} -1 + \frac{1}{4}} \mu^{- \frac{n}{2} + \frac{1}{4}}  \int_{ | \theta - \theta_{0} | \lesssim  (\lambda \mu)^{-1/4}}  \, (\cos{\theta})^{\frac{n}{2}-1+\frac{1}{4}}      d \theta \\
      & \simeq   \lambda^{-\frac{n}{2} + \frac{1}{4}} \mu^{- \frac{n}{2} + \frac{1}{4}}  (\lambda \mu)^{-1/4},
   \end{split}
\end{equation*}
where in the second equivalence we have used $ \psi^{\prime}(t)= \frac{1}{2} \left(  \frac{1}{t} -1 \right)^{1/2} \simeq t^{-1/2}$, for $t \ll 1$ and $x= x(\theta)=\frac{\lambda}{2} \cos{\theta}$, and $\cos{\theta_{0}} \simeq 1 - \frac{\lambda}{2 \mu} \simeq 1$.

For the remaining non-stationary region $| \theta - \theta_{0} | \gtrsim  (\lambda \mu)^{-1/4}  $ inside $[0,  \frac{\pi}{2} - c \lambda^{-1/2}] $ we can use van der Corput lemma to bound
\begin{equation}\label{Non-Stationary Purely Bessel}
   \begin{split}
       & \left| \lambda^{-n+1}  \int_{ | \theta - \theta_{0} | \gtrsim  (\lambda \mu)^{-1/4}} \,   x^{\frac{n}{2}-1} \mu^{-\frac{n}{2}+ \frac{1}{2}  } \, \psi^{\prime}({\textstyle\frac{x}{\mu}})^{-1/2}   \, e^{i \Psi_{\pm}(\theta) }    d \theta \right|\\
       & \simeq \left| \lambda^{-n+1} \mu^{-\frac{n}{2}+ \frac{1}{4}  } \int_{ | \theta - \theta_{0} | \gtrsim  (\lambda \mu)^{-1/4}} \,   x^{\frac{n}{2}- \frac{3}{4}}  \,  e^{i \Psi_{\pm}(\theta) }    d \theta \right|.
   \end{split}
\end{equation}
To this end, since $ \frac{d}{d \theta}(\Psi_{+})$ is decreasing in $\theta$ and $\frac{d}{d \theta}(\Psi_{+})(\theta_{0} \mp (\lambda \mu)^{-1/4}) \sim \frac{d^2}{d \theta^2}(\Psi_{+})(\theta_{0}) (\lambda \mu)^{-1/4}$ so, in view of (\ref{size stationary2}),  in the non-stationary region $| \theta - \theta_{0} | \gtrsim  (\lambda \mu)^{-1/4}  $, we have
\begin{equation*}
    \left| \frac{d}{d \theta}(\Psi_{+})(\theta_{0} \mp (\lambda \mu)^{-1/4}) \right|\gtrsim (\lambda \mu)^{1/4}.
\end{equation*}
We distinguish two cases. If $n\geq2$, then amplitude $\theta \mapsto x(\theta)^{\frac{n}{2} - \frac{3}{4}}$ is decreasing in $\theta$, so from van der Corput lemma, (\ref{Non-Stationary Purely Bessel}) is bounded by
\begin{equation*}
    \lambda^{-n+1} \mu^{-\frac{n}{2}+ \frac{1}{4}} \frac{  (\lambda/2)^{\frac{n}{2} - \frac{3}{4}}  }{   (\lambda \mu )^{1/4}   } \simeq \lambda^{-\frac{n}{2}+ \frac{1}{4}} \mu^{-\frac{n}{2}+ \frac{1}{4}} ( \lambda \mu)^{- \frac{1}{4}},
\end{equation*}
which is acceptable. If $n=1$, then the amplitude $\theta \mapsto x(\theta)^{ - \frac{1}{4}}$ is increasing function of $\theta$ so this time (\ref{Non-Stationary Purely Bessel}) is 
\begin{equation*}
  \begin{split}
       & \left| \mu^{- \frac{1}{4}}  \int_{ | \theta - \theta_{0} | \gtrsim  (\lambda \mu)^{-1/4}} \,   x^{-\frac{1}{4}}   \, e^{i \Psi_{\pm}(\theta) }    d \theta \right|\\
       & \lesssim  \mu^{- \frac{1}{4}} \left[  \frac{  x(\theta_{0} - (\lambda \mu)^{-1/4})^{-\frac{1}{4}}   }{(\lambda \mu)^{1/4}} + \frac{  x(\theta_{0} - c \lambda^{-1/2} )^{-\frac{1}{4}}   }{(\lambda \mu)^{1/4}}   \right]\\
       & \simeq (\lambda \mu)^{- \frac{1}{4}} \mu^{-\frac{1}{4}} \left[ (\lambda)^{-\frac{1}{4}} + (\lambda^{1/2})^{- \frac{1}{4}}\right]\\
       & \simeq (\lambda \mu)^{- \frac{1}{4}} \mu^{-\frac{1}{4}} \lambda^{-\frac{1}{8}},
  \end{split}
\end{equation*}
which is again acceptable. 

It remains to give the desired bound on $R_{k}(\lambda, \sigma)$, for $0 \leq k \leq k_{0}$. Because $k$ are now bounded so  $\mu=2(2k+n) \simeq_{k_{0}} 1$, and therefore $|\lambda| \mu \gg 1$ is equivalent to $|\lambda| \gg 1$. Again, by symmetry, we can assume $\lambda>0$. We decompose $R_{k}(\lambda, \sigma)$ as in (\ref{I plus II Coeff. Rk as Laguerre integral}): $R_{k}(\lambda, \sigma)= R_{k}^{I}(\lambda, \sigma) + R_{k}^{II}(\lambda, \sigma)$. As before the region of integration in $R_{k}^{II}(\lambda, \sigma)$ being the stationary region for the factor $\exp i \frac{1}{4} \lambda \sin{\theta}$ so 
\begin{equation*}\label{x var. II Coeff. Rk as Laguerre integral bounded k}
   \begin{split}
       R_{k}^{II}(\lambda, \sigma) & \sim c_{n} \lambda^{-n+1}  \frac{\Gamma(k+1)}{\Gamma(k+n)} e^{i \frac{1}{4} \lambda}     \int_{0}^{c \lambda^{1/2}}  L_{k}^{n-1}(x) \, e^{-\frac{1}{2} x }  x^{n-1} \frac{dx}{\sqrt{\lambda^2 - 4 x^2}} \\
       & \sim c_{n} \lambda^{-n}  \frac{\Gamma(k+1)}{\Gamma(k+n)}  \,  e^{i \frac{1}{4} \lambda}       \int_{0}^{c \lambda^{1/2}}  L_{k}^{n-1}(x) \, e^{-\frac{1}{2} x }  x^{n-1} dx\\
       & \lesssim_{k_{0},n} \lambda^{-n},
   \end{split}
\end{equation*}
for large $\lambda$, where the last inequality is a consequence of $k \simeq_{k_{0}}1$ and the decay of $e^{-x/2}$ for large $x$. On the other hand, for the piece $R_{k}^{I}(\lambda, \sigma)$, one can substitute the expression of Laguerre polynomial in it and, after swapping summation (in the polynomial) and integration, then applying van der Corput lemma for each term one can easily see $|R_{k}^{I}(\lambda, \sigma)| \lesssim_{k_{0}} \lambda^{-n - N/2}$, for all $N>0$. Combining these two bounds, one obtains $|R_{k}(\lambda, \sigma)| \lesssim_{k_{0}} \lambda^{-n} \simeq_{k_{0}} (\lambda \mu)^{-n}$, in the case $k \leq k_{0}$.

Altogether, this completes the proof of the lemma.

\section{Appendix}\label{appendix}
 The \Cref{all-large-dist-thm Hn} describes the distances of all unbounded sets with a positive upper density. In this section, we initially prove that the sets with higher upper density admit all distances. This is an easy consequence of the quantitative version of the generalized Steinhaus' theorem, which states that all the subsets of the unit ball whose measure is at least $1/2$ admit uniform interval distances. As mentioned in the Introduction, we state and prove these observations in \Cref{uniform distances of small sets}. In addition, we provide an example that explains the necessity of the hypothesis on the upper density by proving  \Cref{prop:Example}. The proofs are trivial analogues from the Euclidean case observed in \cite[Section 10.1]{Pramanik-Raani2023} and we write them out for completion.
 
\subsection{Uniform distances of small sets}\label{uniform distances of small sets proof}
\begin{theorem} \label{uniform distances of small sets} 
 Let $n \geq 1$ and $\rho<\frac12$. 
 There exists a small $c_{\rho} > 0$, depending only on $\rho$ and $n$, such that the following holds: for all Borel sets $E \subseteq B(0,1)\subset \mathbb H^n$ with $|E| \geq 1 -\rho$, we have 
\begin{equation}  \label{Steinhaus-Boardman-Hn}  
E \cdot E^{-1} \supseteq B(0,c_\rho); \quad \text{ as a result } \quad \Delta(E) \supseteq [0,c_{\rho}].  \end{equation} 

As an application, if $C_Q$ denotes the volume of the unit ball $B(0,1)\subset \mathbb H^n$, the sets in Heisenberg group with Kor\'anyi upper density at least $C_Q^{-1}(1-\rho)$ that is, 
\begin{equation}  \limsup_{R \rightarrow \infty} \sup_{x \in \H^n} \frac{|A \cap B(x,R)|}{|B(x,R)|} >C_Q^{-1}(1-\rho),\label{Boardman-condition-Hn} \end{equation} 
has all large distances:
\begin{equation}
\Delta(A) = [0, \infty).\label{part b}
\end{equation}
\end{theorem} 
\begin{proof}
Similar to the argument in \cite[Lemma II]{Boardman-1970}, we prove \eqref{Steinhaus-Boardman-Hn}. Suppose $E \subseteq B(0,1)$, $ |E| \geq 1 -\rho > \frac12$. Since $2(1-\rho) > 1$, we can choose a constant $c_{\rho} >0$ small enough such that
\begin{equation}
    (1+c_{\rho}\sqrt{2n})^{2n}   (1+c_{\rho}^2+c_{\rho}\sqrt{n/2})<2(1-\rho).                             \label{eq:crho}
\end{equation}
Indeed, such a choice is possible since $\rho<\frac12$.
Then for any $x = (z,t) \in [0, c_{\rho}]^{2n}\times [0,c_{\rho}^2]$, we have $$(x\cdot E)\cup E\subset Q_{\rho,n}:= [0,1+c_{\rho}\sqrt{2n}]^{2n}  \times \left[0,1+c_{\rho}^2+c_{\rho}\sqrt{n/2}\right].$$ 
Note that the Haar measure of $Q_{\rho,n}$ is at most $2(1-\rho)$, for the above choice of $c_{\rho}$ in \eqref{eq:crho}. Since $(x\cdot E) \cup E \subseteq Q_{\rho,n}$, we have
\begin{align*} &\bigl|(x \cdot E) \cup E \bigr| \leq |Q_{\rho,n}| < 2(1 - \rho) \leq |x \cdot E| + |E|.     \end{align*}
But this implies $ \bigl| (x \cdot E) \cap E \bigr| > 0$, and hence $(x \cdot E) \cap E$ is non-empty, that is, $x \in E \cdot E^{-1}$.  Since this is true for all $x \in [0, c_{\rho}]^{2n}\times [0,c_{\rho}^2]$, we conclude that $\Delta(E) \supseteq [0, c_{\rho}]$. 
\vskip0.1in
\noindent The proof of \eqref{part b} follows from \eqref{Steinhaus-Boardman-Hn} similar to the proof found in \cite[Section 10.1]{Pramanik-Raani2023}.\\
Given a set $A$ satisfying \eqref{Boardman-condition-Hn}, one can find a sequence $R_k \nearrow \infty$ and $x_k=(z_k,t_k) \in \mathbb H^n$ so that 
\begin{equation}
\frac{|(A \cap B(x_k,R_k)|}{|B(x_k,R_k)|} \geq  C_Q^{-1}(1 - \rho) ,  \quad \text{ for all $k$.} \label{eqdeduc}
\end{equation}
Let us define the transformation $\mathbb T_k: \R^{2n} \times \R \rightarrow \R^{2n} \times \R$ by \[ \mathbb T_k(w,u) := \delta_{R_k^{-1}}\left( x_k^{-1} \cdot (w,u)\right), \]  
where $``\cdot"$ is the group law of $\Ha$. 
Note that $E_k := \mathbb T_k(A \cap B(x_k,R_k))$ and $(z,t)\in E_k$ if and only if $x_{k} \cdot \delta_{R_k}(z,t)\in A\cap B(x_k, R_k)$. We have $E_k\subset B(0,1)$ and $|E_k| \geq 1 -\rho$. Invoking (\ref{Steinhaus-Boardman-Hn}), we obtain $ \Delta(E_k) \supseteq [0, c_{\rho}]$, and hence
\[ \Delta(A) \supseteq \bigcup_{n=1}^{\infty} \Delta \bigl[\mathbb T_k^{-1}(E_k) \bigr] \supseteq \bigcup_{n=1}^{\infty} R_k \Delta(E_k)  \supseteq \bigcup_{k=1}^{\infty} \bigl[0, c_{\rho} R_k\bigr] = [0, \infty), \] 
as claimed. 
\end{proof}

\subsection{Example}\label{Example Rice based}
We shall now turn to the example mentioned in the introduction and show why the positive upper Kor\'anyi density cannot be qualitatively improved. Thus strengthening the hypothesis in our main theorem. This is an easy variant, in the Heisenberg group context, of a parallel example by A. Rice, \cite{2020Rice}, formulated for general metric on $\R^n$. We closely follow the construction given in \cite{2020Rice} and \cite{Pramanik-Raani2023}.

{\emph{Proof of \Cref{prop:Example}}.}
The set  $\mathbb Z^{2n}\times \mathbb Z$ is  discrete with no limit points.  Hence there exists a sequence $ R_m \nearrow \infty$ and $\epsilon_m \searrow 0$ in the following order of choice $ R_1, \epsilon_1,  R_2, \epsilon_2, \ldots$ and obeying the properties: 
\begin{align} 
 \bigl| |(z,t)|_K -  R_j \bigr| > 4\epsilon_m \text{ for all $(z,t) \in \mathbb Z^{2n}\times\mathbb Z$ and } 1 \leq j \leq m,  \label{Rice-conditions-1} \\ \label{Rice-conditions-2}   R_m \geq 100  R_{n-1} \quad \text{ and } \quad f( R_{m})\leq \left(\frac{\epsilon_{m-1}}{16}\right)^2.
 \end{align}  

 By \eqref{Rice-conditions-2}, there exists $(z_m,t_m)$ such that 
 \begin{equation*}
     L_m:=\{(z_m,t_m)\cdot (z,t)^{-1}: (z,t) \in \mathbb Z^{2n}_{\lfloor \frac{ R_m}4\rfloor}\times Z_{\lfloor \frac{ R_m}4\rfloor}\} \subset B\left(0,{\textstyle\frac{ R_m}{2} }\right)\backslash B(0,10 R_{m-1})
 \end{equation*}
Consider $P_m=\{(z,t)\in \mathbb H^n: |(z,t)\cdot (w_m,u_m)^{-1} |_K<\frac14\epsilon_{m-1} \text{ for some }(w_m,u_m)\in L_m\}$ and define $A=\cup_mP_m$. It is easy to see that $$|A\cap B(0, R_m)|\geq f(R_m) \, |B(0, R_m)|.$$

Suppose $A\cap B(0, R_{m-1})$ has no $R_j$ distances for all $j$. By the definition of $P_m$ we have 
\begin{align*}
    d_K((z',t'),(w,u)) \geq  9 R_{m-1},\ \forall (z',t')\in A\cap B(0,  R_{m-1})\text{ and }(w,u)\in P_m
\end{align*}
If there exists $(w_1,u_1),(w_2,u_2)\in P_m$ such that $|(w_1,u_1)\cdot(w_2,u_2)^{-1} |_{K} =  R_j$ for some $1\leq j\leq m-1$ then we have $R_j\leq \frac{\epsilon_{m-1}}2+10 R_{m-1}$, further giving a contradiction to \eqref{Rice-conditions-1}.

\section*{Acknowledgments} 
We thank Malabika Pramanik for helpful suggestions.

\newcommand{\etalchar}[1]{$^{#1}$}
\providecommand{\bysame}{\leavevmode\hbox to3em{\hrulefill}\thinspace}
\providecommand{\MR}{\relax\ifhmode\unskip\space\fi MR }
\providecommand{\MRhref}[2]{%
  \href{http://www.ams.org/mathscinet-getitem?mr=#1}{#2}
}
\providecommand{\href}[2]{#2}

\end{document}